%% file: bare_jrnl12.tex
\documentclass[journal]{IEEEtran}
\ifCLASSINFOpdf
\else
\fi
\input{myIncludes}

\begin{document}
%

\allowdisplaybreaks

\title{\Huge Almost Global Consensus on the $n$-Sphere}
%
%
%

\author{Johan~Markdahl,~\IEEEmembership{Member,~IEEE,}
        Johan~Thunberg,
        and~Jorge~Gon\c{c}alves
\thanks{J. Markdahl, J. Thunberg, and J. Gon\c{c}alves are with the Luxembourg Centre for Systems Biomedicine, University of Luxembourg, Belval, Luxembourg. E-mail: johan.markdahl@uni.lu, johan.thunberg@uni.lu, jorge.goncalves@uni.lu}
\thanks{Manuscript received October X, 2016; revised -, -.}}

%
%

\markboth{IEEE TRANSACTIONS ON AUTOMATIC CONTROL,~VOL.~X, NO.~Y, MONTH~YEAR}%
{Markdahl \MakeLowercase{\textit{et al.}}: Almost Global Consensus on the $n$-Sphere}
%



\maketitle

\begin{abstract}
This paper establishes novel results regarding the global convergence properties of a large class of consensus protocols for multi-agent systems that evolve in continuous time on the $n$-dimensional unit sphere or  $n$-sphere.  For any connected, undirected graph and all $n\in\N\backslash\{1\}$, each protocol in said class is shown to yield almost global consensus. The feedback laws are negative gradients of  Lyapunov functions and one instance generates the canonical intrinsic gradient descent protocol. This convergence result sheds new light on the general problem of consensus on Riemannian manifolds; the $n$-sphere for $n\in\N\backslash\{1\}$ differs from the circle and $\SOT$ where the corresponding protocols fail to generate almost global consensus. Moreover, we derive a novel consensus protocol on $\SOT$ by combining two almost globally convergent protocols on the $n$-sphere for $n\in\{1,2\}$. Theoretical and simulation results suggest that the combined protocol yields almost global consensus on $\SOT$.\end{abstract}


%
\IEEEpeerreviewmaketitle

\section{Introduction}

\noindent \IEEEPARstart{C}{onsider} a network of $N$ agents whose states are points on an $n$-dimensional manifold $\mathcal{M}$. Each agent has a limited capability to sense certain information that pertains to some of the other agents. Distributed control protocols allow a multi-agent system to synchronize its agents, \ie for all agents to reach a consensus as information propagates over time by means of local interactions \cite{mesbahi2010graph}. There are a number of results concerning the case when the initial states of all agents belong to a geodesically convex subset of $\mathcal{M}$  \cite{hartley2013rotation,afsari2013convergence,thunberg2016consensus}, but the likelihood of encountering such a scenario by chance decreases exponentially with $N$. The problem of almost global consensus on Riemannian manifolds is largely unexplored and requires further study \cite{sarlette2009consensus,sarlette2009autonomous}. This paper establishes almost global convergence for a large class of consensus protocols on all \mbox{$n$-spheres} except the circle, a rather unexpected finding.
Consensus problems on the circle and the sphere arise in a number of engineering applications, including cooperative reduced rigid-body attitude control \cite{markdahl2015rigid,song2015distributed}, planetary scale mobile sensing networks \cite{paley2009stabilization}, and self synchronizing chemical and biological oscillators described by the Kuramoto model \cite{kuramoto1975self,dorfler2013synchronization}. 





The reduced attitude provides a model for the orientation of objects that for various reasons, such as task redundancy, cylindrical symmetry, actuator failure, \emph{etc.}, lack one degree of rotational freedom in three-dimensional space. The orientation of such objects corresponds to a pointing direction with the rotation about the axis of pointing being of little to no importance \cite{chaturvedi2011rigid}. The reduced attitude synchronization problem is equivalent to the consensus problem on the $2$-sphere. The problem of cooperative control on the $n$-sphere in $\R^{n+1}$, denoted  $\mathcal{S}^n$, has received some attention in the literature \cite{olfati2006swarms,scardovi2007synchronization,sarlette2009geometry,sarlette2011synchronization,li2014unified,li2015collective,lageman2016consensus} but comparatively less than equivalent problems on $\SOT$ for which there is a considerable literature  \cite{beard2001coordination,lawton2002synchronized,rodriguez2004mutual,sarlette2009autonomous,sarlette2010coordinated,ren2010distributed,tron2013riemannian,hartley2013rotation,tron2014distributed,thunberg2014distributed,Matni2014Convex}. 


The problem of almost global consensus has been studied on $\mathcal{S}^1$ \cite{sarlette2009geometry,sarlette2011synchronization,scardovi2007synchronization}, on $\SOT$ \cite{tron2012intrinsic,ohta2016attitude}, on $\Sn$ in the special case of a complete graph \cite{olfati2006swarms,li2014unified}, and on other Riemannian manifolds \cite{sarlette2009consensus}. Tron \etal \cite{tron2012intrinsic} apply an optimization based method to characterize the stability of all equilibria on $\SOT$ for a particular discrete-time consensus protocol. Their result is akin to almost global consensus over any connected graph topology. The algorithm makes use of a reshaping function which depends on a parameter that must exceed a bound whose value cannot be calculated from local information. Moreover, the overall convergence speed of the algorithm decreases with increasing values of the parameter. In contrast to \cite{tron2012intrinsic}, this paper shows that almost global convergence of a large class of consensus protocols on $\Sn$ for $n\in\N\backslash\{1\}$ can be established without the use of a reshaping function or any non-local knowledge of the graph. Furthermore, we show how this class can be extended to a class of protocols on $\SOT$ that only depend on an upper bound of $N$ and display convergence properties that rival those of \cite{tron2012intrinsic}.

The $2$-sphere is diffeomorphic to the quotient space $\SOT/\mathsf{SO}(2)$ and, as such, many results obtained for $\SOT$ also apply to $\St$. Special cases sometimes allow for stronger results. This paper shows that the conditions for achieving almost global consensus are more favorable on $\Sn$ for $n\in\N\backslash\{1\}$ than what is implied by previously known results concerning $\GC$ and $\SOT$. A large class of intrinsic consensus protocols over connected, undirected graph topologies renders all equilibria but the consensus set unstable on $\Sn$. By contrast, analysis of the corresponding consensus protocols on  $\mathcal{S}^1\simeq\mathsf{SO}(2)$ \cite{sarlette2009geometry,sarlette2011synchronization} and  simulations on $\SOT$ \cite{tron2012intrinsic} show that certain graph topologies yield equilibrium sets aside from consensus that are asymptotically stable on $\SO$ for $n\in\{2,3\}$.

The literature on continuous-time cooperative control on the $n$-sphere has largely been focused on special cases. Previous work either concerns the case of a specific graph topology \cite{olfati2006swarms,li2014unified,song2015distributed,ohta2016attitude}, a specific sphere \cite{olfati2006swarms,scardovi2007synchronization,sarlette2009consensus,sarlette2011synchronization,li2015collective,song2015distributed}, or a specific control law \cite{olfati2006swarms,tron2012intrinsic,song2015distributed,markdahl2016cdc,ohta2016attitude}. Many of them also lack a rigorous proof of almost global convergence \cite{scardovi2007synchronization,sarlette2009consensus,tron2012intrinsic,markdahl2016cdc,ohta2016attitude}. They only show that all equilibrium sets except the consensus set are unstable, which is a weaker result in general \cite{freeman2013global}. We provide a rigorous proof of almost global convergence for a large class of analytic consensus protocols over any connected graph by, roughly speaking, showing that the region of attraction of any set of exponentially unstable equilibria  have measure zero on $(\Sn)^N$. 

In the literature survey \cite{sepulchre2011consensus}, it is observed that almost global convergence of consensus protocols on nonlinear spaces (in particular \GC) is graph dependent. The survey discusses three control design procedures to circumvent this problem: reshaping functions \cite{scardovi2007synchronization,sarlette2009geometry,tron2012intrinsic}, gossip algorithms \cite{sarlette2009geometry}, and dynamic feedback \cite{sarlette2009consensus}. The main contribution of this paper is to show that consensus on $\Sn$ is \emph{not}  graph dependent for any $n\in\N\backslash\{1\}$, and that almost global consensus can be achieved without utilizing any of the three design procedures in \cite{sepulchre2011consensus}. This leads to the contra-intuitive but intriguing notion that almost global consensus is more difficult to achieve on \GC{} than any other sphere. Preliminary results are found in \cite{markdahl2016cdc}, conjectures are made in \cite{olfati2006swarms,markdahl2015rigid}.

\section{Problem Description}
\label{sec:preliminaries}


\noindent The following notation is used in this paper. The inner and outer product of $\ve{x}, \ve{y}\in\R^{n}$  are denoted by $\langle\ve{x},\ve{y}\rangle$ and $\ve{x}\otimes\ve{y}$, respectively. The inner product of $\ma{A},\ma{B}\in\R^{n\times n}$ is  $\langle\ma{A},\ma{B}\rangle=\trace\mat{A}\ma{B}$. Let $\|\cdot\|$ denote the Euclidean norm of a vector and $\|\cdot\|_F$ the Frobenius norm of a matrix. The gradient in Euclidean space is denoted $\nabla:f(\ve{x})\mapsto\nabla f(\ve{x})\in\R^n$, the intrinsic gradient map at a point $\ve{x}$ on a manifold $\mathcal{M}$ is denoted $\inabla:f(\ve{x})\mapsto\inabla f(\ve{x})\in\ts[\mathcal{M}]{\ve{x}}$. We represent manifolds by their canonical embeddings in Euclidean space. The special orthogonal group is $\SO=\smash{\{\ma{R}\in \R^{n\times n}\,|\,\ma{R}\mtr=\ma{R}\inv,\,\det\ma{R}=1\}}$. The Lie algebra of $\SO$ is $\so=\ts[\SO]{\ma{I}}=\{\ma{S}\in\R^{n\times n}\,|\,\ma{S}\mtr=-\ma{S}\}$. The $n$-sphere is $\Sn=\{\ve{x}\in\R^{n+1}\,|\,\|\ve{x}\|=1\}$. The tangent space of $\Sn$ is  $\ts[\Sn]{\ve{x}}=\{\ve{y}\in\R^{n+1}\,|\,\langle\ve{y},\ve{x}\rangle=0\}\simeq\R^{n}$. 
An undirected, simple graph is a pair $\mathcal{G}=(\mathcal{V},\mathcal{E})$ where $\mathcal{V}\subset\N$ is the node set and $\mathcal{E}\subset \{e\subset\V\,|\,|e|=2\}$ is the edge set. A graph $\mathcal{G}$ is said to be connected if it contains a tree subgraph with $|\mathcal{V}|-1$ edges.





\subsection{Distributed Control on the $n$-Sphere}
\label{sec:distributed}

\noindent Consider a multi-agent system where each agent corresponds to an index $i\in\V$ and has a state $\ve[i]{x}\in\Sn$ expressed in a world coordinate frame $\mathcal{W}$. Agent $i$ uses a body-fixed frame $\mathcal{B}_i$  that relates to $\mathcal{W}$ by a rotation matrix $\ma[i]{R}(t)\in\SO$ for all $t\in[0,\infty)$. At each $t$, $\ma[i]{R}(t)$ yields a map  $\ma[i]{R}:[\ve{v}]_{\B}\mapsto[\ve{v}]_{\W}$, where the bracket $[\,\cdot\,]_\mathcal{F}$ denotes that its content is expressed in a frame $\mathcal{F}$. If the frame is omitted, then $\mathcal{W}$ is presupposed. Chose the reduced attitude $\ve[i]{x}$ of agent $i$ to satisfy  $[\ve[i]{x}]_{\B}=\ve[1]{e}$. Thus  $[\ve[i]{x}]_{\W}=\ma[i]{R}[\ve[i]{x}]_{\B}=\ma[i]{R}\ve[1]{e}$, \ie $[\ve[i]{x}]_{\W}$ is given by the first column of $\ma[i]{R}$. The agents are capable of limited local sensing. The topology of the communication network is described by an undirected connected graph $\mathcal{G}=(\mathcal{V},\mathcal{E})$, where $\mathcal{V}=\{i\in\N\,|\,i\leq N\}$, and $\{i,j\}\in \mathcal{E}$ implies that two neighboring agents $i$ and $j$ can sense the so-called relative information $[\mathcal{I}_{ij}]_{\B}$, $[\mathcal{I}_{ji}]_{\mathcal{B}_j}$ regarding the displacement of their states $\ve[i]{x}$ and $\ve[j]{x}$. All relative information agent $i$ has access to is compounded into a set $\mathcal{I}_i$, the precise nature of which may differ between applications.

\begin{system}\label{sys:n}
	The system is given by $N$ agents, an undirected and connected graph $\mathcal{G}=(\mathcal{V},\mathcal{E})$, agent states $\ve[i]{x}\in\Sn$, where $n\in\N$, and dynamics
	\begin{align}\label{eq:sys}
		\vd[i]{x}=\ve[i]{u}-\langle\ve[i]{u},\ve[i]{x}\rangle\ve[i]{x}=(\ma{I}-\ma[i]{X})\ve[i]{u}=\ma[i]{P}\ve[i]{u},
	\end{align}
	where $\ve[i]{u}:\mathcal{I}_{i}\rightarrow\R^{n+1}$ is the input signal of agent $i$, $\ma[i]{X}=\ve[i]{x}\otimes\ve[i]{x}$, and $\ma[i]{P}=\ma{I}-\ma[i]{X}$ for all $i\in\V$.
\end{system}

Control is based on relative information. The information that agent $i$ has access to regarding its neighbor agent $j$ could be defined to include 
\begin{align}
[\pos\{\ve[j]{x}-\ve[i]{x}\}]_{\B}\subseteq[\mathcal{I}_{ij}]_{\B},\label{eq1:relative}
\end{align}
which is the relative information customary to the ambient space $\R^{n+1}$. The set of neighbors of agent $i$ is  $\Ni=\{j\in\V\,|\,\{i,j\}\in\E\}$. The set of relative information known to agent $i\in \mathcal{V}$ is $[\mathcal{I}_i]_{\B}=\pos\cup_{j\in\Ni}[\mathcal{I}_{ij}]_{\B}$. The dynamics \eqref{eq:sys} of agent $i$ projects the input of agent $i$ on the tangent space $\ts[\Sn]{\ve[i]{x}}$, \ie on a hyperplane orthogonal to $\ve[i]{x}$.

\begin{remark}
It can be argued that
\begin{align*} \pos\{\ma[i]{P}(\ve[j]{x}-\ve[i]{x})\}\subseteq\mathcal{I}_{ij},
\end{align*}
where $\ma[i]{P}:\R^{n+1}\rightarrow\ts[\Sn]{\ve[i]{x}}$ is an orthogonal projection matrix, is preferable to \eqref{eq1:relative} since it confines $\mathcal{I}_{ij}$ to an intrinsic rather than an ambient space. However, we believe that the constraints on $\mathcal{I}_{ij}$ tend to come from limited sensing capabilities rather than rigid-body dynamics, and that most applications on $\St$ involve sensors that measure features of ambient rather than intrinsic space. Note that the dynamics \eqref{eq:sys} remain the same in both cases since $\ma[i]{P}^2=\ma[i]{P}$.
\end{remark}

While agent $i$ may not be able to calculate some $[\ve[i]{u}]_{\B}\in[\mathcal{I}_{i}]_{\B}$ based on the information \eqref{eq1:relative} obtained from all its neighbors, that agent may still be able to calculate an input $[\ve[i]{v}]_{\B}\in[\mathcal{I}_{ij}]_{\B}$ whose projection on $\ts[\Sn]{\ve[i]{x}}$ by the dynamics of $\ve[i]{x}$ is identical to that of $\ve[i]{u}$. This holds for  inputs that belongs to $\linspan\cup_{j\in\Ni}\ve[j]{x}$, and in particular for elements of the positive cone $\pos\cup_{j\in\Ni}\ve[j]{x}$. Intuitively speaking, it is reasonable to assume that agent $i$ should be able to sense the bearing and distance to any of its neighbors, and we therefore set $[\mathcal{I}_{i}]_{\B}=[\pos\cup_{j\in\Ni}\{\ve[j]{x}\}]_{\B}$.






The results and proofs in this paper are carried out in the world frame $\W$. To implement the control law in a distributed fashion, $\ve[i]{u}$ must be transfered to $\B$ for all $i\in\V$. Let a control law in $\mathcal{W}$ be given by
$\smash{[\ve[i]{u}]_{\W}=\sum_{j\in\Ni}f_{ij}[\ve[j]{x}]_{\W}}$. Hence $[\ve[i]{u}]_{\B}=\sum_{j\in\Ni}f_{ij}\mat[i]{R}[\ve[j]{x}]_{\W}=\sum_{j\in\Ni}f_{ij}[\ve[j]{x}]_{\B}$. Moreover,
\begin{align}
[\vd[i]{x}]_{\B}&=\mat[i]{R}[\vd[i]{x}]_{\W}=\mat[i]{R}[\ve[i]{u}]_{\W}-\langle[\ve[i]{u}]_{\W},[\ve[i]{x}]_{\W}\rangle\mat[i]{R}[\ve[i]{x}]_{\W}\nonumber\\
&=[\ve[i]{u}]_{\B}-\langle[\ve[i]{u}]_{\B},[\ve[i]{x}]_{\B}\rangle[\ve[i]{x}]_{\B},\label{eq:sysb}
\end{align}
since inner products are invariant under orthogonal changes of coordinates. It is clear from \eqref{eq:sysb} that \eqref{eq:sys} can be implemented in a distributed fashion.

The problem of multi-agent consensus on $\Sn$ concerns the design of distributed control protocols $(\ve[i]{u})_{i=1}^N$ based on relative information that stabilize the consensus set 
\begin{align}
\mathcal{C}&=\{(\ve[i]{y})_{i=1}^N\in(\Sn)^N\,|\,\ve[i]{y}=\ve[j]{y},\,\forall\,i,j\in\mathcal{V}\}\nonumber\\
&=\{(\ve[i]{y})_{i=1}^N\in(\Sn)^N\,|\,\ve[i]{y}=\ve[j]{y},\,\forall\,\{i,j\}\in\mathcal{E}\}\label{eqC:Cmanifold}
\end{align}
of System \ref{sys:n}, where the second equality hinges on the assumption that $\mathcal{G}$ is connected. If the states of all agents assume the same value on the $n$-sphere, then they are said to reach consensus. Terms such as consensus, synchronization, rendezvous, and state-aggregation are used interchangeably in this paper, but note that some authors, see \eg \cite{li2014unified,sarlette2009consensus}, assign the definitions of these concepts subtle nuances.


\subsection{Problem Statement}
\label{secC:problem}

\noindent This paper concerns some aspects of control design but the main focus is stability analysis. Algorithm \ref{algo:constant} is arguably the most basic conceivable feedback for consensus on $\Sn$ by virtue of its correspondence with the linear consensus protocol on $\R^{n+1}$ for single integrator dynamics given by $\vd[i]{x}=\ve[i]{u}$ for all $i\in\mathcal{V}$. Algorithm \ref{algo:constant} is the negative gradient of the Lyapunov function $V=\tfrac{1}{2}\sum_{\{i,j\}\in\E}f_{ij}\|\ve[i]{x}-\ve[j]{x}\|^2$ and generates what may be referred to as the canonical intrinsic gradient descent consensus protocol. As such, it is of interest to determine the limits of Algorithm \ref{algo:constant}'s performance, \ie the global level stability of the consensus set $\mathcal{C}$ as an equilibrium set of System \ref{sys:n}. It is important to establish that the region of attraction of the undesired equilibria is of negligible size, \eg meager in the sense of Baire and of measure zero \cite{freeman2013global}.

\begin{algorithm}\label{algo:constant}
The feedback is given by $\ve[i]{u}=\sum_{j\in\Ni}f_{ij}\ve[j]{x}$, where the constants $f_{ij}\in(0,\infty)$ satisfy $f_{ij}=f_{ji}$ for all $\{i,j\}\in \mathcal{E}$.
\end{algorithm}

\begin{definition}[Measure zero]
	A set $\mathcal{N}\subset(\Sn)^N$ has measure zero if for every chart $\phi:\mathcal{D}\rightarrow\R^{N(n+1)}$ in some atlas of $(\Sn)^N$, it holds that $\phi(\mathcal{D}\cap\mathcal{N})$ has Lebesgue measure zero.
\end{definition}

\begin{definition}[Almost global attractiveness]
	Consider a system that evolves on $(\Sn)^N$.
	A set of equilibria $\mathcal{D}\subset(\Sn)^N$ is said to be almost globally attractive if for all initial conditions $(\ve[i,0]{x})_{i=1}^N\in(\Sn)^N\backslash\mathcal{N}$, where $\mathcal{N}$ is some set of zero measure, it holds that $\lim_{t\rightarrow\infty}(\ve{x}(t))_{i=1}^N\in\mathcal{D}$. \end{definition}

\begin{problem}\label{prob:global}
Show that there is a large class of consensus protocols for System \ref{sys:n}, including Algorithm \ref{algo:constant}, such that the consensus set $\mathcal{C}$ is stable and almost globally attractive.
\end{problem}

Problem \ref{prob:global} concerns the global behavior of System \ref{sys:n}. Under certain assumptions regarding the connectivity of $\mathcal{G}$, local consensus on $\Sn$ can be established with the region of attraction being the largest geodesically convex sets on $\Sn$, \ie open hemispheres \cite{lageman2016consensus}. See also \cite{tron2013riemannian} in the case of an undirected graph and \cite{thunberg2014distributed} in the case of a directed and time-varying graph. A global stability result for discrete-time consensus on $\SOT$ is provided in \cite{tron2012intrinsic}. Almost global asymptotical stability of the consensus set on the $n$-sphere is known to hold when the graph is a tree \cite{tron2013riemannian} or is complete in the case of first- and second-order models \cite{olfati2006swarms,li2014unified}. The author of  \cite{olfati2006swarms} conjectures that global stability also holds for a larger class of topologies whereas \cite{sarlette2009geometry,sarlette2011synchronization} provides counter-examples of basic consensus protocols that fail to generate consensus on $\mathcal{S}^1$.

\begin{remark}
Global consensus on $\Sn$ cannot be achieved by means of a continuous feedback due to topological constraints \cite{bhat2000topological}. It is however possible to achieve almost global asymptotical stability, as has been demonstrated on the circle \cite{sarlette2009geometry,sarlette2011synchronization}. To prove almost global convergence to the consensus set is challenging since basic tools such as the Hartman-Grobman theorem or stable-unstable manifold theorems are unavailable due to the equilibria being nonhyperbolic \cite{sastry1999nonlinear}. Feasible approaches include dual Lyapunov stability theory \cite{rantzer2001dual} and a technique based on stability in the first approximation \cite{freeman2013global} that applies to convergent systems. We  take the latter approach.
\end{remark}

\section{Stability of the Consensus Manifold}

\noindent This section and the next concern System \ref{sys:n} governed by Algorithm \ref{algo:global} which is an extension of Algorithm \ref{algo:constant}. Algorithm \ref{algo:global} provides a large class of smooth continuous-time consensus protocol on the $n$-sphere. The stability properties of all equilibria are fully determined, as is those of the overall system.

\subsection{Control Design}

\noindent Consider a class of consensus protocols that formalizes the idea of increasing system cohesion by moving an agent into the convex hull of its state and those of its neighbors.

\begin{algorithm}\label{algo:global}
	The input is given by
	\begin{align*} \ve[i]{u}=\sum_{j\in\Ni}f_{ij}(s_{ij})\ve[j]{x},
	\end{align*}
	where $s_{ij}=1-\langle\ve[i]{x},\ve[j]{x}\rangle$ and the feedback gains $f_{ij}:\R\rightarrow\R$ are real analytic functions that satisfy 
	\begin{itemize}
		\item[(i)] $f_{ij}>0$,
		\item[(ii)]  $f_{ij}=f_{ji}$,
		\item[(iii)] $(n-2+s_{ij})s_{ij}f_{ij}- (2-s_{ij})s_{ij}^2f_{ij}^\prime>0$,
	\end{itemize}
	for all $s_{ij}\in(0,2]$ and all $\{i,j\}\in\mathcal{E}$.
\end{algorithm}

Note that $f_{ij}$ depends on $s_{ij}:\Sn\times\Sn\rightarrow[0,2]$ given by
\begin{align}\label{eq:sij}
s_{ij}=\tfrac12\|\ve[j]{x}-\ve[i]{x}\|^2=1-\langle\ve[i]{x},\ve[j]{x}\rangle,
\end{align}
which is invariant under orthogonal changes of coordinates. Algorithm \ref{algo:global} therefore complies with the requirements of Section \ref{sec:distributed} regarding distributed feedback laws over the $n$-sphere. Various forms of the closed loop dynamics of System \ref{sys:n} under Algorithm \ref{algo:global} is stated on the readers behalf and for the sake of completeness
\begin{align*}
\vd[i]{x}&=\ve[i]{u}-\langle\ve[i]{u},\ve[i]{x}\rangle\ve[i]{x}\nonumber\\
&=\sum_{j\in\Ni}f_{ij}(s_{ij})\ve[j]{x}-\sum_{j\in\Ni}f_{ij}(s_{ij})\langle\ve[j]{x},\ve[i]{x}\rangle\ve[i]{x}\nonumber\\
&=\sum_{j\in\Ni}f_{ij}(s_{ij})(\ve[j]{x}-(1-s_{ij})\ve[i]{x})\nonumber\\
&=(\ma{I}-\ma[i]{X})\sum_{j\in\Ni}f_{ij}(s_{ij})\ve[j]{x}=\ma[i]{P}\sum_{j\in\Ni}f_{ij}(s_{ij})\ve[j]{x}.
\end{align*}

\begin{remark}
	Algorithm \ref{algo:global} comprises a class of algorithms which includes those of Algorithm \ref{algo:constant} for all $n\in\N\backslash\{1\}$. If $f_{ij}=k\in(0,\infty)$ for all $\{i,j\}\in\mathcal{E}$, then (iii) evaluates to $k(n-2+s_{ij})s_{ij}\geq ks_{ij}^2>0$ for all $s_{ij}\in(0,2]$ when $n\geq2$ but for $n=1$ we obtain 
		\begin{align*}
			(-1+s_{ij})s_{ij}\cdot k+(2-s_{ij})s_{ij}^2\cdot0=-k(1-s_{ij})s_{ij}\leq0
		\end{align*}
	for all $s_{ij}\in[0,1]$. Note that the class grows with $n$. For example, if $f_{ij}=s_{ij}^k$ for some $k\in\N$ then (iii) evaluates to 
	\begin{align*}
	(n-(k+1)(2-s_{ij}))s_{ij}^{k+1},
	\end{align*}
which is positive on $(0,2]$ when $k\leq \tfrac{n}{2}-1$. To see that the class is empty for $n=1$, note that (iii) can be rewritten as
\begin{align}
\frac{f_{ij}^\prime}{f_{ij}}<\frac{-1+s_{ij}}{(2-s_{ij})s_{ij}}\label{eq:frac}
\end{align}
for all $\{i,j\}\in\mathcal{E}$ and all $s_{ij}\in(0,2)$. This implies $\lim_{s_{ij}\rightarrow0}\nicefrac{f_{ij}^\prime}{f_{ij}}=-\infty$. Since $f_{ij}^\prime$ is continuous, it is bounded on $[0,2]$ whereby $f_{ij}(0)=0$ and $f_{ij}^\prime(0)\leq0$. Even if $f_{ij}^\prime(0)=0$, the inequality \eqref{eq:frac} still implies that $f_{ij}^\prime(s)<0$ for all $s\in(0,\delta)$ for some $\delta\in(0,\infty)$. By continuity there exists an $\varepsilon\in(0,\infty)$ such that $f_{ij}(s_{ij})<0$ for all $s_{ij}\in(0,\varepsilon)$, which contradicts requirement (i) of Algorithm \ref{algo:global}.\end{remark}

	%

\begin{remark}
	For some feedback gains $f_{ij}$ there is a ball in the space $\mathcal{C}^\omega$ of real analytic functions consisting entirely of feedback gains of other elements of Algorithm \ref{algo:global}. For instance, Algorithm \ref{algo:constant} still converges if instead of a constant $f_{ij}$ agent $i$ and $j$ use $f_{ij}+g_{ij}$, where $g_{ij}\in\mathcal{C}^\omega$ is of sufficiently small norm. This could be interpreted as a form of robustness against analytic radial errors, \eg constant measurement errors due to biased sensors.
\end{remark}

Algorithm \ref{algo:global} can be derived by taking the gradient of the candidate Lyapunov function $V:(\Sn)^N\rightarrow[0,\infty)$ given by
\begin{align}
	V((\ve[i]{x})_{i=1}^N)&=\sum_{\{i,j\}\in \mathcal{E}}\int^{s_{ij}}_0 f_{ij}(r)\diff r,\label{eq:potential}
\end{align}
where $s_{ij}=1-\langle\ve[i]{x},\ve[j]{x}\rangle$. Let $U:(\R^{n+1})^N\rightarrow\R$ be the extension of $V$ obtained by just changing the domain, \ie 
\begin{align*}
U((\ve[i]{x})_{i=1}^N)&=\sum_{\{i,j\}\in \mathcal{E}}\int^{s_{ij}}_0 f_{ij}(r)\diff r.
\end{align*}
The functions $f_{ij}$ being analytic on $\R$ by assumption implies that $U$ is smooth since integrals of analytic functions are analytic  \cite{boas1996primer}. Denote $\nabla U=(\nabla_{i}U)_{i=1}^N$, where $\nabla_i=\nabla_{\ve[i]{x}}$. Then
\begin{align}\label{eq:gradV}
\nabla_i U&=\sum_{j\in\Ni}\frac{\diff U}{\diff s_{ij}}\nabla_{i}s_{ij}=-\sum_{j\in\Ni}f_{ij}(s_{ij})\ve[j]{x}.
\end{align}
It follows that $\ve[i]{u}=-\nabla_i U$ and $\vd[i]{x}=-\ma[i]{P}\nabla_i U$ for all $i\in\mathcal{V}$.


\begin{proposition}\label{propC:lasalle}
	System \ref{sys:n} under Algorithm \ref{algo:global} converges to an equilibrium set in $(\Sn)^N$. At any equilibrium point, each input is parallel to the state of its agent.
\end{proposition}

\begin{proof}
	Consider the potential function \eqref{eq:potential}. It holds that
	\begin{align}
		\dot{V}&=\sum_{\{i,j\}\in\mathcal{E}}f_{ij}\dot{s}_{ij}=-\sum_{\{i,j\}\in\mathcal{E}}f_{ij}(\langle\vd[i]{x},\ve[j]{x}\rangle+\langle\ve[i]{x},\vd[j]{x}\rangle)\nonumber\\
		&=-\sum_{i\in \V}\left\langle\vd[i]{x},\sum_{j\in\Ni}f_{ij}\ve[j]{x}\right\rangle-\sum_{j\in \V}\left\langle\sum_{i\in\mathcal{N}_j}f_{ij}\ve[i]{x},\vd[j]{x}\right\rangle\nonumber\\
		&=-2\sum_{i\in \V}\left\langle\ve[i]{u}-\langle\ve[i]{u},\ve[i]{x}\rangle\ve[i]{x},\ve[i]{u}\right\rangle\nonumber\\	
		&=-2\sum_{i\in \V}\|\ve[i]{u}\|^2-\langle\ve[i]{u},\ve[i]{x}\rangle^2.\label{eq:cauchy}
	\end{align}
	System \ref{sys:n} converges to the set $\smash{\{(\ve[i]{x})_{i=1}^N\,|\,\dot{V}=0\}}$ by LaSalle's theorem. The Cauchy-Schwarz inequality applied to \eqref{eq:cauchy} shows that the input and state of each agent align up to sign asymptotically. This implies $\vd[i]{x}=\ve{0}$ for all $i\in \mathcal{V}$, \ie that the system is at an equilibrium by inspection of \eqref{eq:sys}.\end{proof}
	
The equilibria that are characterized by Proposition \ref{propC:lasalle} can be divided into three categories:
	\begin{align}
		(\ve[i]{x},\ve[i]{u})\in\left\{\left(-\frac{\ve[i]{u}}{\|\ve[i]{u}\|},\ve[i]{u}\right),\left(\frac{\ve[i]{u}}{\|\ve[i]{u}\|},\ve[i]{u}\right),\left(\ve[i]{x},\ve{0}\right)\right\},\label{eq:equilibria}
	\end{align}
	where $\ve[i]{u}=\sum_{j\in\Ni}f_{ij}\ve[j]{x}$ for all $i\in\V$. The case of $\ve[i]{u}=\ve{0}$ for all $i\in\mathcal{V}$ is illustrated by Figure \ref{figC:iso}. The agent states in Figure \ref{figC:iso} correspond to the six corners of an octahedron, which is one of the five platonic solids. Likewise, the tetrahedral graph (\ie the complete graph over four nodes) has the tetrahedron as an equilibrium with $\ve[i]{x}=\nicefrac{-\ve[i]{u}}{\vn[i]{u}}$ for all $i\in\mathcal{V}$; whereas the cube, icosahedral, and dodecahedral graphs have respectively the cube, icosahedron, and dodecahedron as equilibria with $\ve[i]{x}=\nicefrac{\ve[i]{u}}{\vn[i]{u}}$ for all $i\in\mathcal{V}$.
	\begin{figure}[htb!]
		\centering	\includegraphics[width=0.4\textwidth]{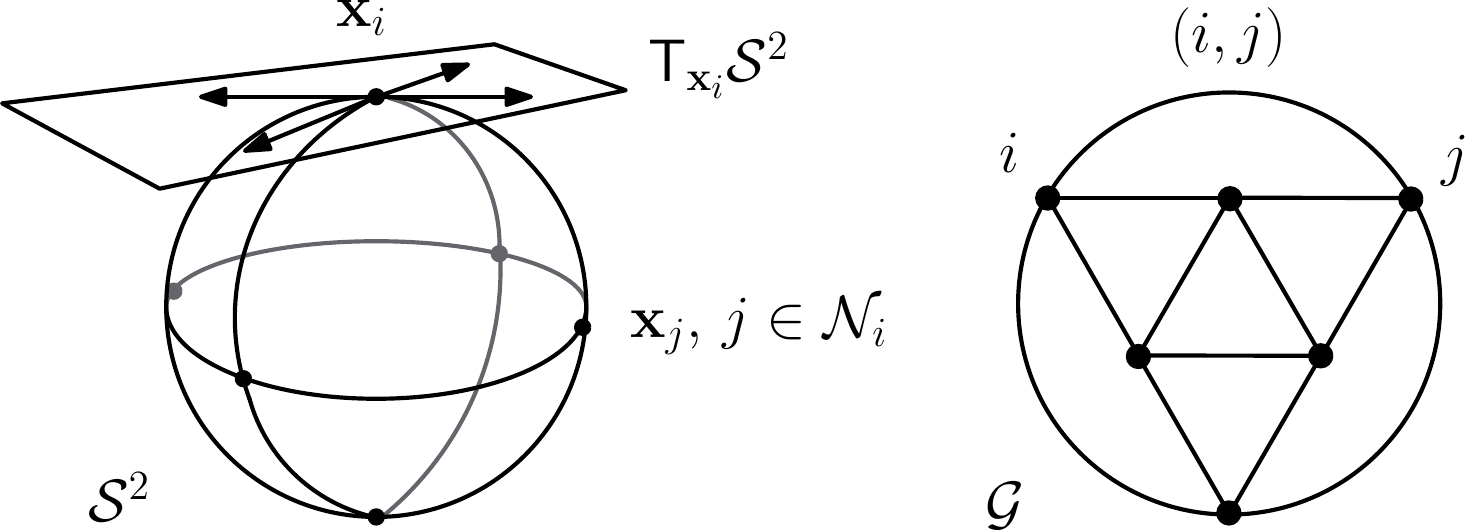}
		\caption{An equilibrium of a system on $\St$ (left) with an octahedral graph (right). The sum of neighbor states is zero, as is the projection of said sum on the tangent plane $\ts[\St]{\ve[i]{x}}$ (left).}
		\label{figC:iso}
	\end{figure}

The following result, Proposition \ref{th:tron}, concerns consensus over the largest geodesically convex sets on $\Sn$, \ie open hemispheres. Analogues to Proposition \ref{th:tron} and various generalizations thereof are known to the control community. For example,  \cite{thunberg2016consensus} uses invariant convex hulls in a manner that was preceded in \cite{afsari2011riemannian,hartley2013rotation} to prove local convergence of time switched consensus protocols on $\mathsf{SE}(3)$. To solve Problem \ref{prob:global}, this paper provides a companion to Proposition \ref{th:tron}, Theorem \ref{th:unique}, which characterizes all equilibrium sets of System \ref{sys:n} under Algorithm \ref{algo:global} in terms of attractiveness and stability.  Although Proposition \ref{th:tron} is used in the proof of Theorem \ref{th:unique}, its full power is not needed. Rather, it is included as a contrast to highlight the greater generality achieved by our analysis.

\begin{proposition}\label{th:tron}
	Consider System \ref{sys:n} under Algorithm \ref{algo:global}. The consensus set $\mathcal{C}$ is asymptotically stable. Moreover, the system reaches consensus asymptotically if there is some finite time such that all agents belong to an open hemisphere.
\end{proposition}

\begin{proof}
	Let $\mathcal{H}$ denote the open hemisphere. Since $f_{ij}\in[0,\infty)$ for all $j\in\Ni$, $\vd[i]{x}=\ma[i]{P}\sum_{j\in\Ni}f_{ij}\ve[j]{x}$ points towards the geodesically convex hull of $\{\ve[j]{x}\,|\,j\in\Ni\}$ on $\Sn$ along the tangent space $\ts[\Sn]{\ve[i]{x}}$ superimposed on $\Sn$ at $\ve[i]{x}$. This shows $\mathcal{H}$ to be invariant and $\mathcal{C}$ to be stable. It remains to show attractiveness. Proposition \ref{propC:lasalle} establishes that System \ref{sys:n} under Algorithm \ref{algo:global} converges to an equilibrium set. Since $\mathcal{H}$ is invariant the desired result follows if the only equilibrium configuration on $\mathcal{H}$ is a consensus.
	
	There must be at least one agent $k$ that minimizes the distance to the boundary of $\mathcal{H}$. At any equilibrium, it holds that $\ve[i]{x}$ is parallel $\ve[i]{u}$ for all $i\in\mathcal{V}$ by Proposition \ref{propC:lasalle}. Since all agents belong to an open hemisphere it follows that $\ve{0},-\ve[i]{x}\notin\pos\{\ve[i]{x}\,|\,i\in\mathcal{V}\}$. By \eqref{eq:equilibria}, only $\ve[i]{x}=\nicefrac{\ve[i]{u}}{\|\ve[i]{u}\|}$ remains. Agent $k$ belongs to an extreme ray of the convex cone $\pos\{\ve[i]{x}\,|\,i\in\mathcal{V}\}$. But then $\ve[k]{x}=\nicefrac{\ve[k]{u}}{\vn[k]{u}}$ if and only if $\ve[j]{x}=\ve[k]{x}$ for all $j\in\mathcal{N}_k$. An induction argument can be applied to show that the system is at a consensus due to $\mathcal{G}$ being a connected graph.\end{proof}

\subsection{Main Result}

\label{sec:main}

\noindent In light of the previous section, we state our main result. 

\begin{theorem}\label{th:unique}
Consider System \ref{sys:n} under Algorithm \ref{algo:global} in the case of $n\in\N\backslash\{1\}$. The consensus set $\mathcal{C}$ given by \eqref{eqC:Cmanifold} is almost globally asymptotically stable. The rate of convergence is locally exponential if the feedback gains $f_{ij}$ are nonzero over $\mathcal{C}$ for all $\{i,j\}\in\mathcal{E}$. Moreover, each trajectory of the system converges to some point. The region of attraction of the set of all unstable equilibria is meager.
\end{theorem}

%
%

The proof of Theorem \ref{th:unique} is given in Section \ref{sec:proof}. Let us briefly sketch  the main ideas. That the consensus set is asymptotically stable follows from Proposition \ref{th:tron}. To prove the exponential  instability of the undesired equilibria we use the indirect method of Lyapunov. The system is linearized around an equilibrium on the $n$-sphere. Perturbing all agents in one direction, \ie towards the consensus set, increases cohesion in one half of the sphere while depleting it in the other half. One such perturbation corresponds to a direction of instability for the linearized system. 
Finally, a known result establishes conditions under which any set of exponentially instable equilibria have a region of attraction that is of measure zero and meager.

\section{Instability of Undesired Equilibrium Sets}
\label{sec:instability}

\noindent The global behavior of the system is determined by the stability and attractiveness of all its equilibria, which often can be characterized locally by means of linearization. To establish almost global convergence  we must show that the set of all unstable equilibria has a region of attraction with measure zero. It is possible for a set of exponentially unstable equilibria to have a region of attraction with non-zero measure, but only if the system fails to be convergent  \cite{freeman2013global}. Our control design guarantees that System \ref{sys:n} under Algorithm \ref{algo:global} is convergent, as shown in Proposition \ref{prop:convergent}.

\subsection{Linearization on the $N$-Fold $n$-Sphere}

\noindent Let us study the signs of the real part of the eigenvalues of the linearization of System \ref{sys:n} under Algorithm \ref{algo:global}. This matrix is also the negative Riemannian Hessian, $\ma{H}=-\inabla^{\,2} V$, of the potential function $V$ given by \eqref{eq:potential}. The Riemannian Hessian of a function can be expressed in terms of its Euclidean gradient, Euclidean Hessian, and the Weingarten map \cite{absil2013extrinsic}. 

\begin{proposition}[P-A. Absil, R. Mahony \& J. Trumpf \cite{absil2013extrinsic}]\label{prop:hessian}
Let $f:\mathcal{M}\rightarrow\R$ be a function defined on a Riemannian submanifold $\mathcal{M}$ of $\R^n$. The intrinsic Hessian map $\inabla^{\,2} f:\mathcal{M}\times\ts[\mathcal{M}]{\ve{x}}\rightarrow\ts[\mathcal{M}]{\ve{x}}:(\ve{x},\ve{w})\mapsto\inabla^{\,2}f(\ve{x})\ve{w}$ is given by
\begin{align*}
\inabla^{\,2} f(\ve{x})\ve{w}=\ma{P}\nabla^{\,2} g(\ve{x})\ve{w}+W(\ve{w},(\ma{I}-\ma{P})\nabla g(\ve{w})),
\end{align*}
where $\ma{P}:\mathcal{M}\times\R^n\rightarrow\ts[\mathcal{M}]{\ve{x}}$
 is an orthogonal projection, $g:\R^n\rightarrow\R^n$ is a smooth extension of $f$ to $\R^n$, and $W:\ts[\mathcal{M}]{\ve{x}}\times(\ts[\mathcal{M}]{\ve{x}})^\perp\rightarrow\ts[\mathcal{M}]{\ve{x}}$ is the Weingarten map. 
\end{proposition}
 
\begin{proposition}\label{prop:linearization} The blocks of the linearization matrix $\ma{H}=(\ma[ij]{H})\in\R^{N(n+1)\times N(n+1)}$ of System \ref{sys:n} under Algorithm \ref{algo:global} are given by 
	\begin{align*}
		\ma[ii]{H}&=
		-\langle\ve[i]{u},\ve[i]{x}\rangle\ma[i]{P}-\sum_{j\in\Ni}f_{ij}^\prime\ma[i]{P}\ma[j]{X}\ma[i]{P}
	\end{align*}
	for all $i\in\mathcal{V}$,
	\begin{align*}
		\ma[ij]{H}&=\ma[i]{P}\left(f_{ij}\ma{I}-f_{ij}^\prime\ve[j]{x}\otimes\ve[i]{x}\right)\ma[j]{P}
	\end{align*}
	for all $\{i,j\}\in\mathcal{E}$, and $\ma[ij]{H}=\ma{0}$ otherwise. 
\end{proposition}

\begin{proof}
We use the technique of Proposition \ref{prop:hessian}. The Euclidean gradient is $\nabla U=[\nabla_i U]$ where $\nabla_iU=-\ve[i]{u}$ as seen by \eqref{eq:gradV}. The Euclidean Hessian is $[\nabla^{\,2}_{ji}U]$, where $\nabla^{\,2}_{ji}U\in\R^{(n+1)\times(n+1)}$ is given by
\begin{align*}
\nabla_{ji}^{\,2}U=\begin{cases}
f_{ij}^\prime\ma[j]{X}& \textrm{ if } j=i,\\
-f_{ij}\ve{I}+f_{ij}^\prime\ve[j]{x}\otimes\ve[i]{x} & \textrm{ if } j\in\Ni,\\
\ve{0} & \textrm{ otherwise,}
\end{cases}
\end{align*}
as can be seen by calculation. The projection $\ma{P}$ is a block-diagonal matrix whose $i$th block is given by $\ma[i]{P}=\ma{I}-\ma[i]{X}$. The $i$th block of the matrix $\ma{I}-\ma{P}$ is hence $\ma[i]{X}$. The Weingarten map $W$ at $(\ve[i]{x})_{i=1}^N\in(\Sn)^N$ is given by 
\begin{align*}
W([\ve[i]{t}],[\ve[i]{n}])=[M(\ve[i]{t},\ve[i]{n})],
\end{align*}
where $M$ is the Weingarten map on $\Sn$. The Weingarten map at a point $\ve{x}\in\Sn$ is derived in \cite{absil2013extrinsic} as
\begin{align*}
M(\ve{t},\ve{n})=-\langle\ve{n},\ve{x}\rangle\ve{t},
\end{align*}
where $\ve{t}\in\ts[\Sn]{\ve{x}}$ and $\ve{n}\in(\ts[\Sn]{\ve{x}})^\perp$.

By Proposition \eqref{prop:hessian}, the intrinsic Hessian on $(\Sn)^N$ can be expressed as a block matrix $[\inabla^{\,2}_{\,ij}V]$, where $\inabla_i$ denote $\inabla_{\ve[i]{x}}$, which satisfies
\begin{align*}
\sum_{j\in\V}\inabla^{\,2}_{\,ij}V\ve[j]{w}={}&\ma[i]{P}\nabla^{\,2}_{ii}U\ve[i]{w}+\sum_{j\in\Ni}\ma[i]{P}\nabla^{\,2}_{ji}U\ve[j]{w}+\\
&M(\ve[i]{w},(\ma{I}-\ma[i]{P})\nabla_i U).
\end{align*}
Since $\ve[i]{w}\in\ts[\Sn]{\ve[i]{x}}$, it holds that $\ve[i]{w}=\ma[i]{P}\ve[i]{v}$ for some $\ve[i]{v}\in\R^{n+1}$, whereby
\begin{align*}
\sum_{j\in\V}\inabla^{\,2}_{\,ij}V\ve[j]{w}={}&f_{ij}^\prime\ma[i]{P}\ma[j]{X}\ma[i]{P}\ve[i]{v}+\\
&\ma[i]{P}\sum_{j\in\Ni}(-f_{ij}\ma{I}+f_{ij}^\prime\ve[j]{w}\otimes\ve[i]{x})\ma[j]{P}\ve[j]{v}-\\
&\langle-\ma[i]{X}\ve[i]{u},\ve[i]{x}\rangle\ma[i]{P}\ve[i]{v}\\
={}&(\langle\ve[i]{u},\ve[i]{x}\rangle\ma[i]{P}+f^\prime_{ij}\ma[i]{P}\ma[j]{X}\ma[i]{P})\ve[i]{v}-\\
&\ma[i]{P}\sum_{j\in\Ni}f_{ij}(\ma{I}-f_{ij}^\prime\ve[j]{x}\otimes\ve[i]{x})\ma[j]{P}\ve[j]{v}.
\end{align*}
This equation gives the Riemannian Hessian, $\inabla^{\,2} V$, by inspection; negating it gives the linearization matrix. The blocks on the diagonal are symmetric, and the off-diagonal blocks satisfy $\inabla^{\,2}_{\,ji}V=\inabla^{\,2}_{\,ij}V^\top$ as we would expect from a Hessian.
\end{proof}

\subsection{Instability of Undesired Equilibria}


\noindent Consider an equilibrium such that all agents belong to the intersection of $\mathcal{S}^n$ and a hyperplane in $\R^{n+1}$. Perturb all agents into an open hemisphere by an arbitrarily small movement along a direction orthogonal to the hyperplane. By Proposition  \ref{th:tron}, the perturbed system converges to a  consensus. 
The spectral properties of a linearized system determine how it reacts to perturbations. This is the basic idea behind Proposition \ref{prop:pos}: perturb all agents in the same direction, \eg towards the north pole. This increases cohesion in the north hemisphere while depleting it in the south. We show that one such perturbation corresponds to a direction of exponential instability. 

\begin{proposition}\label{prop:pos}
	Any equilibrium $(\ve[i]{x})_{i=1}^N\notin\mathcal{C}$ of System \ref{sys:n} under Algorithm \ref{algo:global} is exponentially unstable.
\end{proposition}

\begin{proof}
	
	The proof makes use of the linearization provided by Proposition \ref{prop:linearization}. The Courant-Fischer-Weyl min-max principle bounds the range of the Rayleigh quotient of a symmetric matrix by its minimal and maximal eigenvalues \cite{horn2012matrix}. If the Rayleigh quotient is positive for some argument, then the maximal eigenvalue is positive. Recall that if $\ma{H}$ has a positive eigenvalue at an equilibrium, then that equilibrium is unstable by the indirect method of Lyapunov \cite{khalil2002nonlinear}.	
	
	Let $\ve{v}=[\vet{y} \ldots \vet{y}]\mtr\in\ts[(\Sn)^N]{\mathcal{C}}$, \ie $\ve{y}\in\R^{n+1}$ since $\cup_{\ve{x}\in\Sn}\ts[\Sn]{\ve{x}}\simeq\R^{n+1}$, and consider
	\begin{align*}
		\langle\ve{v},\ma{H}\ve{v}\rangle={}&\sum_{i\in\V}\langle\ve{y},\ma[ii]{H}\ve{y}\rangle+\sum_{j\in\Ni}\langle\ve{y},\ma[ij]{H}\ve{y}\rangle\\
		={}&\left\langle\ve{y},\left(\sum_{i\in\V}	\ma[ii]{H}\right.\right.+\left.\left.\sum_{j\in\Ni}\ma[ij]{H}\right)\ve{y}\right\rangle.
	\end{align*}
	Denote $\ma{G}=\sum_{i\in\V}\ma[ii]{H}+\sum_{j\in\Ni}\ma[ij]{H}$. The matrix $\ma{G}$ is symmetric since
	\begin{align*}
	\ma{G}&=\sum_{i\in\V}\ma[ii]{H}+\hspace{-2mm}\sum_{(i,j)\in\E}\hspace{-1mm}\ma[ij]{H}+\ma[ji]{H}=\sum_{i\in\V}\ma[ii]{H}+\hspace{-2mm}\sum_{(i,j)\in\E}\hspace{-1mm}\ma[ij]{H}+\mat[ij]{H}
	\end{align*}
	wherefore $\sigma(\ma{G})\subset\R$ by the spectral theorem. If $\ma{G}$ has a strictly positive eigenvalue, then for the corresponding eigenvector  $\ve{z}\in\R^{n+1}$ it holds that $\langle\ve{z},\ma{G}\ve{z}\rangle>0$ whereby setting $\ve{y}=\ve{z}$ yields $\langle\ve{v},\ma{H}\ve{v}\rangle>0$. The min-max principle then  implies that $\ma{H}$ has a strictly positive eigenvalue, \ie the equilibrium is exponentially unstable.
	
	
	
	Let us prove that $\ma{G}$ has a positive eigenvalue. Consider 
	\begin{align*}
		\trace\ma{G}={}&\sum_{i\in\V}-n\langle\ve[i]{u},\ve[i]{x}\rangle+\sum_{j\in\Ni}\left(-f_{ij}^\prime(1-\langle\ve[i]{x},\ve[j]{x}\rangle^2)+\right.\\
		{}&\left.f_{ij}( n-1+\langle\ve[i]{x},\ve[j]{x}\rangle^2)-f_{ij}^\prime\langle\ve[i]{x},\ve[j]{x}\rangle(\langle\ve[i]{x},\ve[j]{x}\rangle^2-1)\right)\\
		={}&\sum_{i\in\V}-n\langle\ve[i]{u},\ve[i]{x}\rangle+\sum_{j\in\Ni}\left(f_{ij}( n-1+\langle\ve[i]{x},\ve[j]{x}\rangle^2)-\right.\\
		&\left.f_{ij}^\prime (2-s_{ij})s_{ij}^2\right)\\
		={}& n\left(\sum_{i\in\V}-\langle\ve[i]{u},\ve[i]{x}\rangle+\sum_{j\in\Ni}f_{ij}\right)-\\
		&\sum_{i\in\V}\sum_{j\in\Ni}f_{ij}(2-s_{ij})s_{ij}+f_{ij}^\prime (2-s_{ij})s_{ij}^2\\
		={}& n\left(\sum_{i\in\V}\sum_{j\in\Ni}-f_{ij}\langle\ve[i]{x},\ve[j]{x}\rangle+\sum_{j\in\Ni}f_{ij}\right)-\\
		&\sum_{i\in\V}\sum_{j\in\Ni}f_{ij}(2-s_{ij})s_{ij}+f_{ij}^\prime (2-s_{ij})s_{ij}^2\\
		={}& n\sum_{i\in\V}\sum_{j\in\Ni}f_{ij}s_{ij}-\\
		{}&\sum_{i\in\V}\sum_{j\in\Ni}f_{ij}(2-s_{ij})s_{ij}+f_{ij}^\prime (2-s_{ij})s_{ij}^2\\
		={}&\sum_{i\in\V}\sum_{j\in\Ni}f_{ij}(n-2+s_{ij})s_{ij}-f_{ij}^\prime (2-s_{ij})s_{ij}^2,
	\end{align*}
	where we used that $\trace\ma[i]{X}=\|\ve[i]{x}\|^2=1$ and $s_{ij}=1-\langle\ve[i]{x},\ve[j]{x}\rangle$. Recall that
	\begin{align*}
		(n-2+s_{ij})s_{ij}f_{ij}-(2-s_{ij})s_{ij}^2f_{ij}^\prime>0
	\end{align*}
	for all $s_{ij}\in(0,2]$ and all $\{i,j\}\in\mathcal{E}$ by condition (iii) of Algorithm \ref{algo:global}. Since $\trace\ma{G}\geq0$ with strict inequality unless $s_{ij}=0$ for all $\{i,j\}\in\E$, \ie unless $(\ve[i]{x})_{i=1}^N\in\mathcal{C}$, it follows that $\ma{G}$ has a strictly positive eigenvalue.\end{proof}

\begin{remark}
	Requirement (iii) in Algorithm \ref{algo:global} arises from the lower bound on the largest eigenvalue of $\ma{H}$ implied by the sign of $\trace \ma{G}$. This lower bound is likely to be conservative with respect to the requirements on $f_{ij}$ for all $\{i,j\}\in\mathcal{E}$ that results in $\ma{H}$ having a positive eigenvalue. The class of control signals that yield almost global consensus on $\Sn$ should hence be larger than  Algorithm \ref{algo:global}.
\end{remark}


Proposition \ref{prop:unstable} is used to prove Theorem \ref{th:unique}. The version presented here is particularized for our purposes; a more general result and its proof may \eg be found in \cite{freeman2013global}.

\begin{proposition}[R.\um{}A. Freeman \cite{freeman2013global}]\label{prop:unstable}	Consider a system $\vd{x}=\ve{f}(\ve{x})$ that evolves on a state-space $\mathcal{X}$, where $\ve{f}\in\smash{\mathcal{C}^1}$. Let $\mathcal{S}\subset\mathcal{X}$ be a set consisting entirely of exponentially unstable equilibria. If each trajectory of the system converges to some equilibrium, then the region of attraction of $\mathcal{S}$ is of zero measure and meager in $\mathcal{X}$.
\end{proposition}



\subsection{Point-Wise Convergence}

\noindent The instability requirements of Proposition \ref{prop:unstable} are satisfied by Proposition \ref{prop:pos}. However, to show that every trajectory of the system converges to a point, \ie that the system is so-called pointwise convergent \cite{lageman2007convergence}, requires some additional analysis. Point-wise convergence is of importance since Proposition \ref{propC:lasalle} only establishes convergence to equilibrium sets, all of which have $n$ degrees of rotational invariance. In theory, it would be possible for each agent to traverse its sphere indefinitely: each agent would move along a path of rotational invariance of the full agent configuration, while the system as a whole approaches an equilibrium set. The use of Proposition \ref{prop:analytic}, a corollary of the \L{}ojasiewicz gradient inequality \cite{lojasiewicz}, may be not be necessary but suffices to establish point-wise convergence. This is the reason that we assume the feedback gains $f_{ij}\in \mathcal{C}^\omega$ for all $\{i,j\}\in\mathcal{E}$ rather than $f_{ij}\in\mathcal{C}^1$.

%




\begin{proposition}[S. \L{}ojasiewicz \cite{lojasiewicz,lageman2007convergence}]\label{prop:analytic}
Let $\mathcal{M}$ be a real analytic Riemannian manifold and $f : \mathcal{M} \rightarrow\R$ be a real analytic function. 
For the Riemannian gradient flow $\vd{x}=-\inabla f$ it either holds that $\lim_{t\rightarrow\infty}\ve{x}(t)=\ve{y}$ for some $\ve{y}\in\mathcal{M}$ or the set of $\omega$-limit points is empty. 
\end{proposition}

\begin{proposition}\label{prop:convergent}
Each trajectory of System \ref{sys:n} under Algorithm \ref{algo:global} converges to an equilibrium.
\end{proposition}

\begin{proof}
The $n$-sphere is a real analytic manifold, and so is $\SnN$. Sums, composite functions, integrals, and derivatives of multivariate analytic functions are analytic \cite{boas1996primer}. By analyticity of the feedback gains in Algorithm \ref{algo:global}, it follows that the candidate Lyapunov function $V$ given by \eqref{eq:potential} is analytic.


Equation \eqref{eq:gradV} only provides the extrinsic gradient $\nabla U:(\R^{n+1})^N\rightarrow(\R^{n+1})^N$ of \eqref{eq:potential} without regard to the fact that $(\ve[i]{x})_{i=1}^N\in\SnN$. The intrinsic gradient $\inabla V:(\Sn)^N\rightarrow\ts[(\Sn)^N]{\,}$ is given by
\begin{align*}
\inabla V&=(\inabla_{\,i}\!V)_{i=1}^N=(\ma[i]{P}\nabla_i U)_{i=1}^N,
\end{align*}
where $\inabla_{\,i}=\inabla_{\,\ve[i]{x}}$. The intrinsic gradient $\inabla V$ is hence the projection of $\nabla V$ on the tangent space $\ts[\SnN]{(\ve[i]{x})_{i=1}^N}$  \cite{absil2009optimization}. Equation \eqref{eq:gradV} gives $\nabla_i U=-\ve[i]{u}$ whereby 
\begin{align*}
\inabla V=-\left(\ma[i]{P}\ve[i]{u}\right)_{i=1}^N.
\end{align*}
The closed-loop dynamics of System \ref{sys:n} under Algorithm \ref{algo:global} can be written
\begin{align}\label{eq:descent}
\vd[i]{x}=-\inabla_{\,i}\!V
\end{align}
for all $i\in\mathcal{V}$, \ie it is an intrinsic gradient descent flow on $(\Sn)^N$.

The conditions of Proposition \ref{prop:analytic} are satisfied by $\SnN$ and \eqref{eq:descent}. Since the canonical embedding of $\SnN$ in $(\R^{n+1})^N$ is compact, every sequence has a convergent subsequence by the Bolzano-Weierstrass theorem. The set of limit points is hence nonempty. It follows that $(\ve[i]{x})_{i=1}^N$ converges to a single point, and by Proposition \ref{propC:lasalle} that point is an equilibrium.\end{proof}

\subsection{Proof of Main Theorem}
\label{sec:proof}

\noindent Recall that it remains to prove Theorem \ref{th:unique}. Proposition \ref{th:tron}, \ref{prop:pos}, \ref{prop:unstable}, and \ref{prop:convergent} provide the sufficient tools to do so. 

\begin{proof}[Proof of Theorem \ref{th:unique}]
The requirements of Proposition \ref{prop:unstable} are satisfied by Proposition \ref{prop:convergent} and Proposition \ref{prop:pos}. Since all system trajectories converge to equilibria by Proposition \ref{prop:convergent}, and the set of initial conditions resulting in trajectories that converge to any equilibrium that does not belong to the consensus set is of zero measure and meager by Proposition \ref{prop:unstable}, it follows that the set of trajectories converging to the consensus set is almost all of $(\Sn)^N$. This establishes almost global attractiveness. Stability follows from Proposition \ref{th:tron}.

It remains to show local exponential stability. The linearized system dynamics expressed in the variables $(\ve[i]{y})_{i=1}^N$ when $\ve[i]{x}=\ve{c}\in\Sn$, $\ma[i]{P}=\ma{I}-\ve{c}\otimes\ve{c}$ for all $i\in\mathcal{V}$ are hence
\begin{align}\label{eq:linearization}
\vd[i]{y}=(\ma{I}-\ve{c}\otimes\ve{c})\sum_{j\in\Ni}f_{ij}(0)(\ve[j]{y}-\ve[i]{y}).
\end{align}
Each vector $\ve[i]{y}$ of the linearized system evolves along a hyperplane of codimension 1 given by  $\mathcal{H}=\ts[\Sn]{\ve[i]{x}}=\mathrm{Im}(\ma{I}-\ve{c}\otimes\ve{c})$ for all $i\in\mathcal{V}$. Since the graph is connected, and $f_{ij}(0)$ is strictly positive for all $\{i,j\}\in\mathcal{E}$, it follows that \eqref{eq:linearization} reaches consensus exponentially if $\ve[i]{y}(0)\in\mathcal{H}$ for all $i\in\mathcal{V}$ \cite{mesbahi2010graph}.\end{proof}

\section{Perspectives}

\noindent Let us compare what is known with regard to consensus on $\mathcal{S}^1$ and $\SOT$ in relation to Theorem \ref{th:unique}.

\subsection{The Circle and the Sphere}
\label{sec:circlesphere}

\noindent Algorithm \ref{algo:constant} does not satisfy property (iii) of Algorithm \ref{algo:global} in the case of $n=1$. This requirement is however only sufficient for almost global consensus. A counter-example that rules out almost global convergence is provided by \cite{sarlette2009geometry,sarlette2011synchronization}: the equilibrium set over cycle graphs where agents are spread out equidistantly over \GC{} such that the geodesic distance $d_\theta:\GC\times\GC\rightarrow[0,\pi]$ satisfies $d_\theta(\ve[i]{x},\ve[j]{x})=\nicefrac{2\pi}{N}$ for all $\{i,j\}\in\mathcal{E}$ is asymptotically stable. This section explores the difference between $\GC$ and $\St$ with regard to the preconditions for achieving almost global consensus.

\begin{example} Consider six agents on $\St$ and a cycle graph 
\begin{align*}
\mathcal{G}=(\{i\in\N\,|\,i\leq6\},\{1,6\}\cup\{\{i,j\}\in\V\times\V\,|\,i-j=1\})
\end{align*}
where we use the weights $f_{ij}=1$ for all $\{i,j\}\in\mathcal{E}$. One equilibrium consists of the agents being equidistantly spread out over a great circle at a geodesic distance $\nicefrac{\pi}{3}$ from one another, see Figure \ref{fig:cycle}. The linearization matrix is
\begin{align*}
\ma{H}&=\ma{P}(\ma{C}\otimes\ma{I})\ma{P},\\
\ma{P}&=\begin{bmatrix}
\ma[1]{P} & \ma{0} & \ma{0} & \ma{0} & \ma{0} & \ma{0}\\
\ma{0} & \ma[2]{P} &\ma{0} & \ma{0} & \ma{0} & \ma{0}\\
\ma{0} & \ma{0} & \ma[3]{P} & \ma{0} & \ma{0} & \ma{0}\\
\ma{0} & \ma{0} & \ma{0} & \ma[4]{P} & \ma{0} & \ma{0}\\
\ma{0} & \ma{0} & \ma{0} & \ma{0} & \ma[5]{P} & \ma{0}\\
\ma{0} & \ma{0} & \ma{0} & \ma{0} & \ma{0} & \ma[6]{P}
\end{bmatrix},\\
\ma{C}&=\begin{bmatrix}
-1 &  \phantom{-}1 &  \phantom{-}0 &  \phantom{-}0 &  \phantom{-}0 &  \phantom{-}1\phantom{-}\\
\phantom{-}1 & -1 &  \phantom{-}1 &  \phantom{-}0 &  \phantom{-}0 &  \phantom{-}0\phantom{-}\\
\phantom{-}0 &  \phantom{-}1 & -1 &  \phantom{-}1 &  \phantom{-}0 &  \phantom{-}0\phantom{-}\\
\phantom{-}0 &  \phantom{-}0 &  \phantom{-}1 & -1 &  \phantom{-}1 &  \phantom{-}0\phantom{-}\\
\phantom{-}0 &  \phantom{-}0 &  \phantom{-}0 &  \phantom{-}1 & -1 &  \phantom{-}1\phantom{-}\\
\phantom{-}1 &  \phantom{-}0 &  \phantom{-}0 &  \phantom{-}0 &  \phantom{-}1 & -1\phantom{-}
\end{bmatrix},
\end{align*}
where $\otimes$ denotes the Kronecker product. The block diagonal elements of $\ma{P}$ satisfy $\ma[i]{P}=\ma{I}-\ma{R}^{i-1}\ve{e}\otimes\ve{e}(\ma{R}\mtr)^{i-1}$ for some $\ma{R}\in\SOT$ and $\ve{e}\in\St$. Note that $\ma{C}$ is a circulant matrix for which all eigenpairs can be calculated explicitly \cite{davis1979circulant}.

\begin{figure}[htb!]
	\centering	\includegraphics[width=0.35\textwidth]{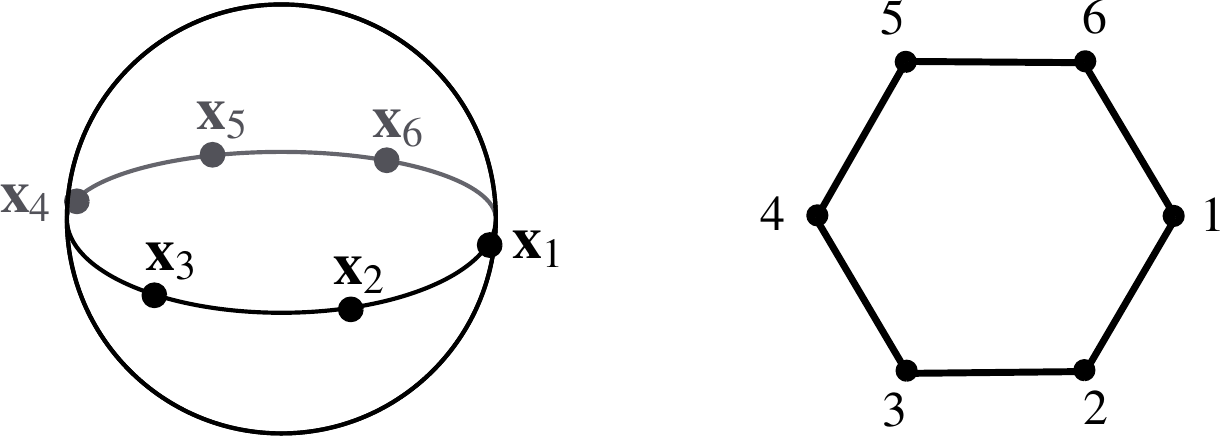}
	\caption{An equilibrium of a system on $\St$ with a cycle graph.}
	\label{fig:cycle}
\end{figure}

There is no loss of generality in setting $\ve{e}=\ve[1]{e}=[1\,\,0\,\,0\thinspace]\mtr$ and positioning all agents on the equator to decouple $\ma[i]{P}$ into a block diagonal matrix,
\begin{align*}
\ma[i]{P}&=\begin{bmatrix}
\ma[i]{Q} & \ve{0}\\
\ve{0} & 1
\end{bmatrix}=\ma{I}-\ma{R}^{i-1}\ve[1]{e}\otimes\ve[1]{e}(\ma{R}\mtr)^{i-1},\\
\ma{R}&=\tfrac12\begin{bmatrix} 1 & \!-\sqrt{3} &0\\
\sqrt{3} & \!\phantom{-} 1 &0\\
0&\!\phantom{-}0&1
\end{bmatrix}.
\end{align*}
The linearized dynamics are thereby decoupled into two independent subsystems, which we represent using the variables $\ve[i]{y}\in\R^2$ and $z_i\in\R$ for all $i\in\V$. For perturbations $(\ve[i]{y})_{i=1}^N$ that belong to the equatorial plane, it follows that 
\begin{align}\label{eq:equator}
\vd[i]{y}&=\ma[i]{Q}(\ma[i-1]{Q}\ve[i-1]{y}+\ma[i+1]{Q}\ve[i+1]{y}-\ve[i]{y}),
\end{align}
for all $i\in\mathcal{V}$, where the indices are added modulo 6. For perturbations that are normal to said plane, it holds that 
\begin{align*}
\vd{z}=\ma{C}\ve{z}
\end{align*}
where $\ve{z}=[z_1\,z_2,\ldots,z_6\,]\mtr$.

The dynamics \eqref{eq:equator} can be written on the form 
\begin{align*}
	\vd{y}&=\ma{Q}(\ma{C}\otimes\ma{I})\ma{Q}\ve{y},
\end{align*}
where $\ma{Q}$ is a block diagonal matrix with $\ma[i]{Q}$ for $i\in\mathcal{V}$ as blocks, $\otimes$ denote the Kronecker product, $\ma{I}$ is the identity matrix of dimension $2$, and $\ve{y}=[\vet[1]{y},\ldots,\vet[6]{y}]\mtr\in\R^{2N}$. Unlike $\sigma(\ma{P}(\ma{C}\otimes\ma{I})\ma{P})$, the spectrum of $\ma{Q}(\ma{C}\otimes\ma{I})\ma{Q}$ belongs to the closed left-half complex plane. 
An eigenpair can be interpreted as a perturbation direction of the system resulting in an instantaneous response that is either aligned or negatively aligned with the perturbation.  For example, 
\begin{align*}
(0,[\vet{v}\!,\,(\ma{T}\ve{v})\mtr\!,\,(\ma{T}^2\ve{v})\mtr\!,\,(\ma{T}^3\ve{v})\mtr\!,\,(\ma{T}^4\ve{v})\mtr\!,\,(\ma{T}^5\ve{v})\mtr]\mtr),
\end{align*}
where $\ma{T}\in\mathsf{SO}(2)$ is the leading principal submatrix of $\ma{R}$,
is an eigenpair of $\ma{Q}(\ma{C}\otimes\ma{I})\ma{Q}$ for all $\ve{v}\in\R^2$. It corresponds to the perturbation of moving each agent a fixed distance along its tangent space, thereby rotating the entire cyclic formation.

The dynamics of $\ve{z}$ are unstable since $(1,[1\,1\,1\,1\,1\,1]\mtr)$ is an eigenpair of $\ma{C}$. This eigenpair can be interpreted as a perturbation that takes all agents into the north hemisphere, from where they reach consensus at the north pole. Another eigenpair is $(-3,[1\,{-1}\,\,\,1\,{-1}\,\,\,1\,{-1}]\mtr)$. The corresponding perturbation lifts and drops agents above and below the equator, thereby distancing any agent from the convex hull of itself and its neighbors. The response is hence a recoil towards the equator, as demonstrated by the negative eigenvalue. The effects of both these perturbations on the original nonlinear system are illustrated in Figure \ref{fig:perturbations}. 
\begin{figure}[htb!]
	\centering	\includegraphics[width=0.42\textwidth]{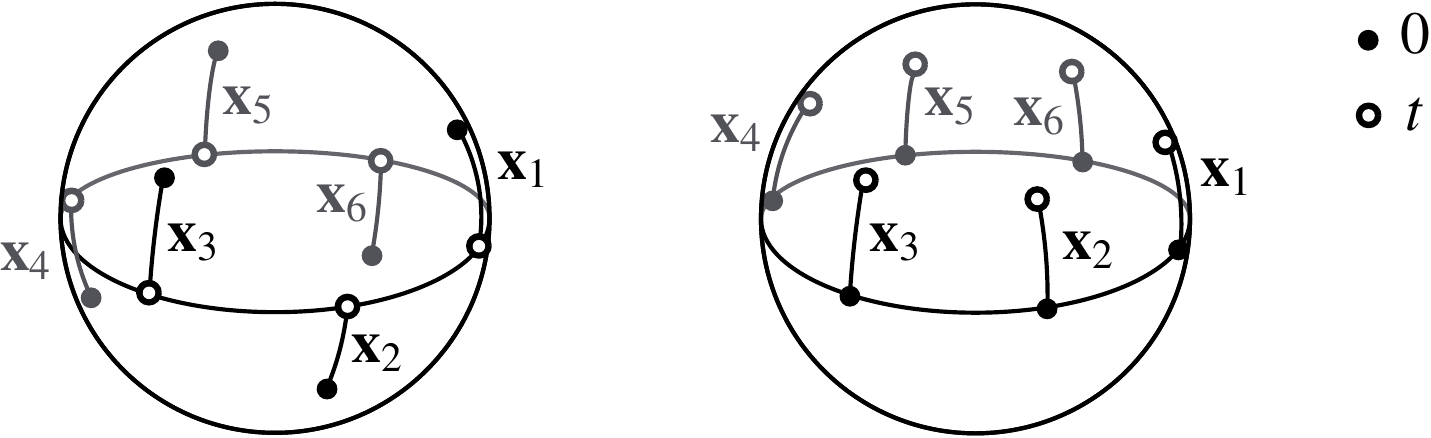}
	\caption{The trajectories of two nonlinear systems which are perturbed from an equilibrium at the equator along the stable and unstable manifolds (left and right respectively).}
	\label{fig:perturbations}
\end{figure}
\end{example}

The unstable directions of perturbations are all orthogonal to the equator. The stability of a cycle equilibrium of System \ref{sys:n} under Algorithm \ref{algo:constant} on $\mathcal{S}^1$ is therefore not inherited by the embedding of $\mathcal{S}^1$ in higher dimensional spheres. Aside from the instability, it is important to note such a perturbation bring all agents into a hemisphere from where they reach consensus by Proposition \ref{th:tron}. This implies that the equator is  unattractive. The circle can also be embedded on an infinite cylinder, but that case is not covered by this analysis.

The following corollary of Theorem \ref{th:unique} lack the generality of its precursor, but is nevertheless a result that we find to be interesting in its own right. It provides an exhaustive characterization of the stability properties of a particular dynamical system, both forwards and backwards in time. Recall that $s_{ij}$ defined by \eqref{eq:sij} measures the extrinsic distance between two points on $\Sn$. Theorem  \ref{col:duality} states that, under certain conditions, Algorithm \ref{algo:constant} solves both the minimax and maximin problems of $s_{ij}$ over all $\{i,j\}\in\mathcal{E}$ almost globally.
\begin{theorem}\label{col:duality}
Consider System \ref{sys:n} on $\St$ under Algorithm \ref{algo:constant} with $f_{ij}=1$ for all $\{i,j\}\in\mathcal{E}$, where $\mathcal{G}=(\mathcal{V},\mathcal{E})$ is a cycle graph. The $\alpha$-limits of the flow from almost all initial conditions belong to
\begin{align*}
\{(\ve[i]{x})_{i=1}^N\in\SnN|\,s_{ij}=\max_{\SnN}\min_{\{k,l\}\in\mathcal{E}}s_{kl},\forall\,\{i,j\}\in\mathcal{E}\}
\end{align*}
whereas the $\omega$-limits belong to the consensus set, \ie
\begin{align*}
\{(\ve[i]{x})_{i=1}^N\in\SnN|\,s_{ij}=\min_{\SnN}\max_{\{k,l\}\in\mathcal{E}}s_{kl},\,\forall\,\{i,j\}\in\mathcal{E}\}.
\end{align*}
\end{theorem}
\begin{proof}
This is a direct consequence of Theorem \ref{th:unique} and the characterization of equilibria obtained by closing System \ref{sys:n} with the negation of Algorithm \ref{algo:constant} provided in \cite{markdahl2015rigid,song2015distributed}.
\end{proof}





\subsection{Simulations}
\label{sec:simulation}

\noindent This section compares the global performance of two consensus protocols on $\Sn$ for $n\in\{1,2\}$ and on $\SOT$ respectively in simulation. To that end, consider the following multi-agent system on the special orthogonal group $\SO$.

\begin{system}\label{sys:SOT}
	The system is given by $N$ agents, an undirected graph $\mathcal{G}=(\mathcal{V},\mathcal{E})$, agent states $\ve[i]{R}\in\SO$, and dynamics $\md[i]{R}=\ma[i]{\Omega}\ma[i]{R}$ where $\ma[i]{\Omega}\in\so$ for all $i\in \mathcal{V}$. It is assumed that $\mathcal{G}$ is connected and that the system can be actuated on a kinematic level, \ie $\ma[i]{\Omega}$ is the input signal of agent $i$.
\end{system}

Recall that Algorithm \ref{algo:constant} can be derived by taking the Riemannian gradient of the potential function  \eqref{eq:potential}. A related consensus protocol on $\SO$ can be derived by taking the Riemannian gradient of the potential function $V:(\SO)^N\rightarrow[0,\infty)$ given by
\begin{align*}
	V((\ve[i]{R})_{i=1}^N)=\tfrac12\sum_{\{i,j\}\in \mathcal{E}}\int^{s_{ij}}_0f_{ij}(r)\diff r,
\end{align*}
where $f_{ij}:[0,2n]\rightarrow[0,\infty)$ and $s_{ij}=n-\langle\ma[i]{R},\ma[j]{R}\rangle$ for all $\{i,j\}\in\mathcal{E}$. As such, Algorithm \ref{algo:constant} is similar to the following algorithm on System \ref{sys:SOT}. 


\begin{algorithm}\label{algo:SO3}
The feedback is given by 
\begin{align*}
\ma[i]{\Omega}=\sum_{j\in\Ni}f_{ij}(s_{ij})(\mat[i]{R}\ma[j]{R}-\mat[j]{R}\ma[i]{R}),
\end{align*}
where $f_{ij}=f_{ji}$ for all $\{i,j\}\in \mathcal{E}$.\end{algorithm}

Table \ref{tab:hexa} displays the outcome of running $10^4$ trials of Algorithm \ref{algo:constant} on System \ref{sys:n} and $10^4$ trials of Algorithm \ref{algo:SO3} on System \ref{sys:SOT} for three different graphs (we set $f_{ij}=5$ for all $\{i,j\}\in \mathcal{E}$ for both algorithms). The initial conditions are drawn uniformly from the sphere using the fact that $\ve{x}\in\mathcal{N}(\ve{0},\ma{I})$ implies that $\nicefrac{\ve{x}}{\|\ve{x}\|}\in \mathcal{U}(\Sn)$ \cite{muller1959note}. This method is also used to draw from $\mathcal{U}(\SOT)$ by first generating a uniform distribution on the unit sphere in quaternion space, \ie drawing from $\mathcal{U}(\mathcal{S}^3)$, and then mapping the sample to $\SOT$. By inspection of Table \ref{tab:hexa}, note that Algorithm \ref{algo:constant} fails to yield almost global consensus on \GC. Likewise, almost global consensus does not hold for Algorithm \ref{algo:SO3} on System \ref{sys:SOT} over $\SOT$. These results agree with those of \cite{sarlette2009geometry,tron2012intrinsic}. As predicted by Theorem \ref{th:unique}, there were no failures to reach consensus on $\St$ despite the high number of trials.


\begin{table}[h]
	\centering
	\normalsize
	\caption{Number of failures to reach consensus on the space $\mathcal{X}\in\{\GC,\St,\SOT\}$ over $10^4$ random trials using Algorithm \ref{algo:constant} and \ref{algo:SO3} with constant feedback gains.} 
	\begin{tabular}{cccc} 
		 $\mathcal{X}$ &  \includegraphics[width=0.041\textwidth]{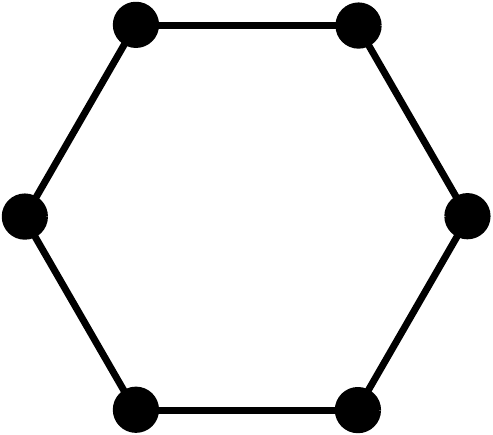} & \includegraphics[width=0.078\textwidth]{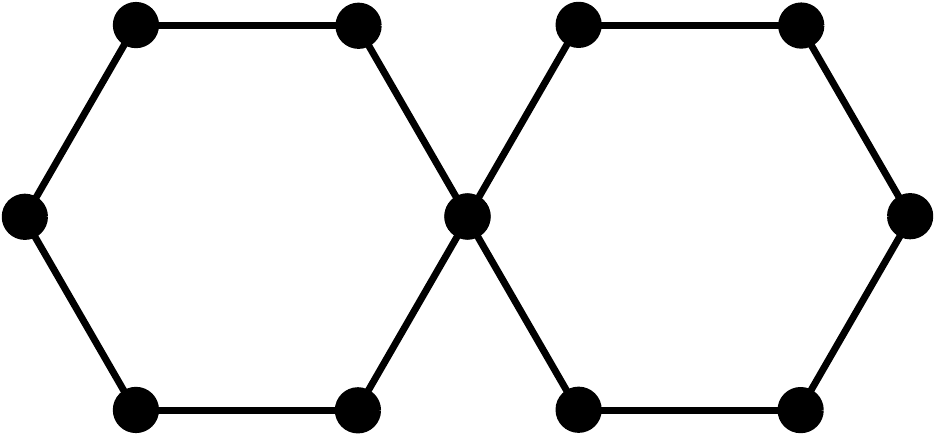} &  \includegraphics[width=0.068\textwidth]{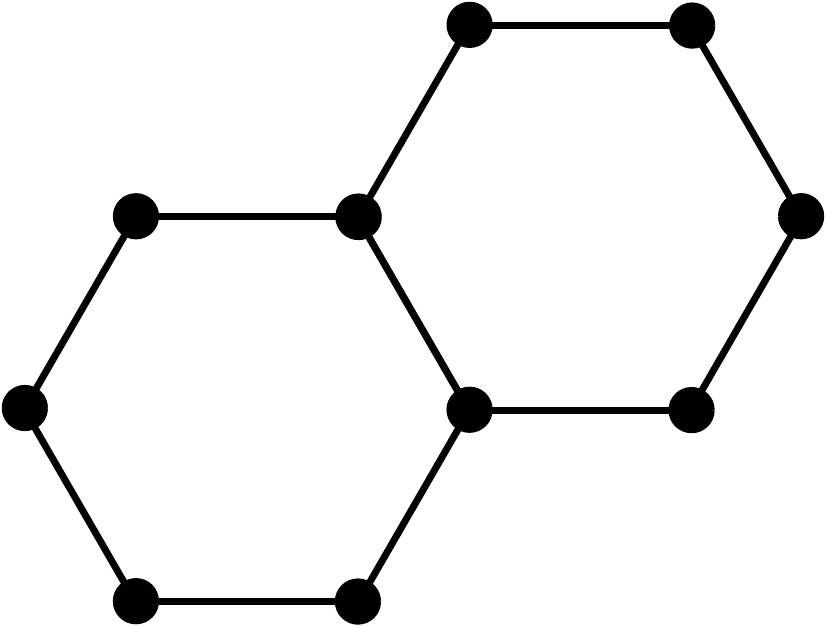}\\
		\cmidrule{1-4}
		$\GC$ &  1504 & 2173 & 2126\\ 
		$\St$ &  $0$ & $0$ & $0$\\
		$\,\,\SOT$ & 711 & 66 & 86 
	\end{tabular}
	\label{tab:hexa}
\end{table}

\subsection{Extension to the Special Orthogonal Group}
\label{sec:extension}

\noindent In Section \ref{sec:circlesphere} we learn that a certain undesired equilibrium set of System \ref{sys:n} under Algorithm \ref{algo:constant} on $\mathcal{S}^1$ is stable. Section \ref{sec:simulation} shows that the problem of multi-agent consensus on $\SOT$ poses similar challenges. In fact, if the reduced attitudes of all agents agree, then the remaining degree of rotational freedom of each agent is confined to a set that is diffeomorphic to $\mathcal{S}^1$. On the $n$-sphere, a perturbation that is orthogonal to the equator will allow a system in such a configuration to reach consensus. On $\SOT$, the destabilizing effect of such a perturbation is counter-acted by the reduced attitude which, figuratively speaking, serves as a ballast that stabilizes the two other axes of all agents to a single great circle.

Let us utilize what we have learned about consensus on $\mathcal{S}^1$ and $\mathcal{S}^2$ to attempt to design a control law on $\SOT$ that stabilizes the consensus set almost globally. To that end, rewrite the variables $\ma[i]{R}$ of \eqref{sys:SOT} as $\ma[i]{R}=[\ve[i]{x}\,\ve[i]{y}\,\ve[i]{z}]$, \ie $\ma[i]{x}=\ma[i]{R}\ve[1]{e}$, $\ma[i]{y}=\ma[i]{R}\ve[2]{e}$, and $\ma[i]{z}=\ma[i]{R}\ve[3]{e}$ whereby 
\begin{align*}
\md[i]{R}=[\vd[i]{x}\,\vd[i]{y}\,\vd[i]{z}]=\ma[i]{\Omega}[\ve[i]{x}\,\ve[i]{y}\,\ve[i]{z}]
\end{align*}
for all $i\in\mathcal{V}$. Let $\ma{S}:\R^3\rightarrow\sot$ be the bijective linear map defined by $\ma{S}(\ve{x})\ve{y}\mapsto\ve{x}\times\ve{y}$ for all $\ve{x},\ve{y}\in\R^3$. Denote $\ve[i]{\omegaup}=\ma{S}\inv(\ma[i]{\Omega})$ for all $i\in\V$. The following algorithm decouples the evolution of $\ve[i]{x}$ from any dependence on $\ve[i]{y}$  and $\ve[i]{z}$ by utilizing a decomposition of $\ve[i]{\omegaup}$ into a part that is orthogonal to $\ve[i]{x}$ and a part that is parallel to $\ve[i]{x}$.

\begin{algorithm}\label{algo:SO32}
	The feedback is given by
	\begin{align*} \ma[i]{\Omega}&=\ma{S}\left(\ve[i]{x}\times\ve[i]{u}+\sum_{j\in\Ni}g_{ij}\ve[i]{x}\right),
	\end{align*}
	where $\ve[i]{u}$ is the input signal of Algorithm \ref{algo:global} and the locally Lipschitz function $g_{ij}:\mathcal{I}_i\rightarrow\R$ is related to the feedback gain of an almost globally convergent consensus protocol on $\GC$ for all $\{i,j\}\in\E$. More specifically, we require that the feedback gains $g_{ij}$ are such that the system
	\begin{align}
	\vd[i]{y}=\sum_{j\in\Ni}g_{ij}\ve[i]{z}, \quad \vd[i]{z}=-\sum_{j\in\Ni}g_{ij}\ve[i]{y},\label{sys:s1}
	\end{align}
	reach consensus for almost all initial conditions such that $\ve[i]{x}(0)=\ve[j]{x}(0)$ for all $\{i,j\}\in\E$ (these dynamics evolve over a single great circle on $\St$ since $\ve[i]{z}=\ve[i]{x}\times\ve[i]{y}$ and $\ve[i]{x}$ is constant, \ie $\ve[i]{z}$ can be expressed in terms of $\ve[i]{y}$ for all $i\in\V$).	
\end{algorithm}

\begin{remark}
To see that Algorithm \ref{algo:SO32} can be implemented by only using local and relative information, note that 
\begin{align*}
[\ma[i]{\Omega}]_{\B}&=\mat[i]{R}[\ma[i]{\Omega}]_{\W}\ma[i]{R}=\ma{S}\left(\mat[i]{R}(\ve[i]{x}\times\ve[i]{u})+\hspace{-1mm}\sum_{j\in\Ni}g_{ij}\mat[i]{R}\ve[i]{x}\right)\\
&=\ma{S}\left((\mat[i]{R}\ve[i]{x})\times(\mat[i]{R}\ve[i]{u})+\sum_{j\in\Ni}g_{ij}\mat[i]{R}\ve[i]{x}\right)\\
&=\ma{S}\left(\ve[1]{e}\times\sum_{j\in\Ni}\mat[i]{R}\ma[j]{R}\ve[1]{e}+\sum_{j\in\Ni}g_{ij}\ve[1]{e}\right).
\end{align*}
The feedback hence only only depends on the relative information $(\mat[i]{R}\ma[j]{R})_{j\in\Ni}$ on $\SOT$. \end{remark}

The closed loop dynamics of System \ref{sys:n} under Algorithm \ref{algo:SO32} are given by
\begin{align}
\vd[i]{x}&=(\ve[i]{x}\times\ve[i]{u})\times\ve[i]{x}=\ve[i]{u}-\langle\ve[i]{u},\ve[i]{x}\rangle\ve[i]{x}=\ma[i]{P}\ve[i]{u},\label{eq:x}\\
\vd[i]{y}&=(\ve[i]{x}\times\ve[i]{u})\times\ve[i]{y}+\sum_{j\in\Ni}g_{ij}\ve[i]{z},\nonumber\\
&=-\langle\ve[i]{u},\ve[i]{y}\rangle\ve[i]{x}+\sum_{j\in\Ni}g_{ij}\ve[i]{z},\label{eq:y}\\
\vd[i]{z}&=(\ve[i]{x}\times\ve[i]{u})\times\ve[i]{z}-\sum_{j\in\Ni}g_{ij}\ve[i]{y},\nonumber\\
&=-\langle\ve[i]{u},\ve[i]{z}\rangle\ve[i]{x}-\sum_{j\in\Ni}g_{ij}\ve[i]{y},\label{eq:z}
\end{align}
for all $i\in\mathcal{V}$. 

Note that any implementation of Algorithm  \ref{algo:SO32} involves the use of an almost globally convergent consensus protocol on \GC, \eg that of \cite{sarlette2009geometry,sarlette2011synchronization}. The protocol of \cite{sarlette2009geometry,sarlette2011synchronization} requires an upper bound on the total number of agents, which is a weaker form of graph dependence than that of the protocol in \cite{tron2012intrinsic}. The following result establishes that it is possible to fuse Algorithm \ref{algo:global} with the protocol of \cite{sarlette2009geometry,sarlette2011synchronization} in a manner which retains Lipschitz continuity.

\begin{proposition}
	The	class of feedbacks laws described by Algorithm \ref{algo:SO32} is nonempty.
\end{proposition}
\begin{proof}
	We need to show that there exists at least one consensus protocol $g_{ij}$ with the required properties. Let
	\begin{align*}
		g_{ij}=g\left(\acos\left(\frac{\langle\ve[i]{y},\ve[j]{y}\rangle}{(\langle\ve[i]{y},\ve[j]{y}\rangle^2+\langle\ve[i]{z},\ve[j]{y}\rangle^2)^{\frac12}}\right)\sgn\langle\ve[i]{z},\ve[j]{y}\rangle\right),
	\end{align*}
	where $g$ is the almost globally convergent consensus protocol for the dynamics on $\GC$ in
	\cite{sarlette2009geometry,sarlette2011synchronization}, \ie
	\begin{align*}
		g(\vartheta)=\begin{cases}
			-\tfrac{1}{N-1}(\pi+\vartheta) & \textrm{ if }\vartheta\in[-\pi,-\tfrac1N\pi),\\
			\vartheta & \textrm{ if } \vartheta\in[-\tfrac1N\pi,\tfrac1N\pi],\\
			\tfrac{1}{N-1}(\pi-\vartheta) & \textrm{ if }\vartheta\in(\tfrac1N\pi,\pi].
		\end{cases}
	\end{align*}
	To see that $g_{ij}$ is Lipschitz, note that the discontinuity of the sign function appears when $\langle\ve[j]{y},\ve[i]{z}\rangle=0$ in which case the argument of $g$ is $\acos\sgn\langle\ve[i]{y},\ve[j]{y}\rangle\in\{0,\pi\}$ and $g(-\pi)=g(\pi)$. 
	
	Suppose that $\ve[i]{x}=\ve[j]{x}$ for all $\{i,j\}\in\E$. Let $\{ \ve[1]{v},\ve[2]{v}\}$ be an orthonormal basis of the plane $\mathcal{P}$ such that $\ve[i]{y},\ve[i]{z}\in\mathcal{P}$ for all $i\in\V$. If $(\ma[i]{R}(0))_{i=1}^N\in\mathcal{S}_2$,  then 
	\begin{align*}
		\ve[i]{y}&=\cos\vartheta_i \ve[1]{v}+\sin\vartheta_i\ve[2]{v},\\
		\ve[i]{z}&=\cos(\vartheta_i+\tfrac{\pi}{2})\ve[1]{v}+\sin(\vartheta_i+\tfrac{\pi}{2})\ve[2]{v}\\
		&=-\sin\vartheta_i\ve[1]{v}+\cos\vartheta_i\ve[2]{v}
	\end{align*}
	for some $\vartheta_i\in(-\pi,\pi]$ for all $i\in\mathcal{V}$. Moreover, 
	\begin{align*}
		\vd[i]{y}&=-\dot{\vartheta}_i\sin\vartheta_i\ve[1]{v}+\dot{\vartheta}_i\cos\vartheta_i\ve[2]{v}=\dot{\vartheta}_i\ve[i]{z},
	\end{align*}
	wherefore \eqref{eq:y} yields
	\begin{align*}
		\dot{\vartheta}_i&=\langle\ve[i]{z},\vd[i]{y}\rangle.=\left\langle\ve[i]{z},-\langle\ve[i]{u},\ve[i]{y}\rangle\ve[i]{x}+\sum_{j\in\Ni}g_{ij}\ve[i]{z}\right\rangle=\sum_{j\in\Ni}g_{ij}.
	\end{align*}

	Note that $\ve[j]{y}=\langle\ve[i]{y},\ve[j]{y}\rangle\ve[i]{y}+\langle\ve[i]{z},\ve[j]{y}\rangle\ve[i]{z}$ on $\mathcal{P}$. The argument of $g$ is hence 
	$\acos\langle\ve[i]{y},\ve[j]{y}\rangle\sgn\langle\ve[i]{z},\ve[j]{y}\rangle=\vartheta_j-\vartheta_i$, which can be interpreted as a signed relative arc length on $\GC$. As such, the dynamics on of $\vartheta_i$ reduces to
	\begin{align*}
		\dot{\vartheta}_i&=\sum_{j\in\Ni}g(\vartheta_j-\vartheta_i),
	\end{align*}
	\ie to the form of the almost globally convergent consensus protocol \cite{sarlette2009geometry,sarlette2011synchronization} on $\GC$.
\end{proof}

Note that the dynamics of $(\ve[i]{x})_{i=1}^N$ given by \eqref{eq:x} are precisely those of System \ref{sys:n} under Algorithm \ref{algo:global}. The consensus set for the reduced attitudes $(\ve[i]{x})_{i=1}^N$ is hence almost globally asymptotically stable by Theorem \ref{th:unique}. We will utilize the triangular structure of the system given by \eqref{eq:x}--\eqref{eq:z} to establish a local convergence result. To this end, consider Proposition \ref{prop:reduction} from \cite{el2013reduction} which have been adapted to our setting.



\begin{proposition}[M.I. El-Hawwary \& M. Maggiore \cite{el2013reduction}]\label{prop:reduction}
Consider a system $\vd{x}=\ve{f}(\ve{x})$, where $\ve{f}$ is locally Lipschitz, that evolves on a compact state-space $\mathcal{X}$. Let $\mathcal{S}_1$ and $\mathcal{S}_2$, where $\mathcal{S}_1\subset\mathcal{S}_2\subset\mathcal{X}$, be two closed, positively invariant
sets. Then, $\mathcal{S}_1$ is asymptotically stable if the following conditions hold:
\begin{itemize}
\item[(i)] $\mathcal{S}_1$ is asymptotically stable relative to $\mathcal{S}_2$,
\item[(ii)] $\mathcal{S}_2$ is asymptotically stable.
\end{itemize}
\end{proposition}

\begin{remark}
There exists a global version of Proposition \ref{prop:reduction} \cite{el2013reduction}. The case when convergence from $\mathcal{S}_2$ to $\mathcal{S}_1$ is global but convergence from $\mathcal{X}$ to $\mathcal{S}_2$ is almost global can be addressed by redefining $\mathcal{X}$ to be the region of attraction of $\mathcal{S}_2$, see \cite{roza2014class}. However, it cannot be applied in our situation. The problem is that $\mathcal{S}_1$ is only almost globally stable relative to $\mathcal{S}_2$. We cannot guarantee that the convergence from $\mathcal{X}$ to $\mathcal{S}_2$ would not bring the system to a state at which convergence from $\mathcal{S}_2$ to $\mathcal{S}_1$ fails. 
\end{remark}

\begin{proposition}\label{col:SOT}
The consensus set on $\SOT$,
\begin{align*}
\mathcal{C}=\{(\ma[i]{R})_{i=1}^N\in(\SOT)^N\,|\,\ma[i]{R}=\ma[j]{R},\,\forall\,\{i,j\}\in\E\},
\end{align*}
is an asymptotically stable equilibrium set of System \ref{sys:SOT} under Algorithm \ref{algo:SO32}.
\end{proposition}

\begin{proof}
In terms of Proposition \ref{prop:reduction}, let 
\begin{align*}
\mathcal{S}_1&=\{(\ma[i]{R})_{i=1}^N\in(\SOT)^N\,|\,\ma[i]{R}=\ma[j]{R},\,\forall\,\{i,j\}\in\mathcal{E}\},\\
\mathcal{S}_2&=\{(\ma[i]{R})_{i=1}^N\in(\SOT)^N\,|\,\ma[i]{R}\ve[1]{e}=\ma[j]{R}\ve[1]{e},\,\forall\,\{i,j\}\in\mathcal{E}\},
\end{align*}
denote the consensus set and reduced attitude consensus set. Clearly $\mathcal{S}_1$, $\mathcal{S}_2$ are closed, positively invariant, nested sets. Property (ii) follows by application of Theorem \ref{th:unique} to the dynamics \eqref{eq:x}. To establish property (i), consider the case of $(\ma[i]{R}(0))_{i=1}^N\in\mathcal{S}_2$. Then $\ve[i]{y},\ve[i]{z}\in\mathcal{P}$ for all $i\in\V$, where $\mathcal{P}$ is the plane that has $\ve[i]{x}$ as normal for any $i\in\mathcal{V}$. The system \eqref{eq:x}--\eqref{eq:z} is hence on the form \eqref{sys:s1}. The consensus set of the system \eqref{sys:s1} is almost globally asymptotically stable by our assumptions on $g_{ij}$ for all $\{i,j\}\in\mathcal{E}$, which implies (i).\end{proof}

Let us return to the simulation problem of Section \ref{sec:simulation}. Generating uniformly distributed initial conditions on $\SOT$ and simulating Algorithm \ref{algo:SO32} where the algorithm of \cite{sarlette2009geometry,sarlette2011synchronization} is used to generate consensus on $\GC$ for the three graph topologies of Table \ref{tab:hexa}, we find no failures to reach consensus.
Algorithm \ref{algo:SO32} hence outperforms Algorithm \ref{algo:SO3} and rivals the practical performance of the algorithm in \cite{tron2012intrinsic}. Moreover, the version of Algorithm \ref{algo:SO32} based on \cite{sarlette2009geometry,sarlette2011synchronization} only requires each agent to know an upper bound on $N$. Algorithm \ref{algo:SO32} also rivals the theoretical performance of \cite{tron2012intrinsic}, as shown in Theorem \ref{prop:instability} of Theorem \ref{th:unique}. Note that we cannot conclude that the consensus manifold is almost globally stable from the result of Theorem \ref{prop:instability} since System \ref{sys:SOT} under Algorithm \ref{algo:SO32} is not a gradient descent flow.

\begin{theorem}\label{prop:instability}
Suppose all feedback gains $g_{ij}$, for $\{i,j\}\in\E$, in Algorithm \ref{algo:SO32} are chosen such that all equilibria of system \eqref{sys:s1} are exponentially unstable except for those in $\mathcal{C}$. Then all equilibria of System \ref{sys:SOT} under Algorithm \ref{algo:SO32} except those in $\mathcal{C}$ are  exponentially unstable. Moreover, $\mathcal{C}$ is asymptotically stable.
\end{theorem}

\begin{proof}
Note that the linearization decouples like the dynamics \eqref{eq:x}--\eqref{eq:z}. Theorem \ref{th:unique} establishes that the all equilibria except those belonging to the consensus set are unstable for the subsystem \eqref{eq:x}. Any candidate for a stable equilibrium must hence satisfy $\ve[i]{x}=\ve[j]{x}$ for all $\{i,j\}\in\E$. This requirement reduces the dynamics \eqref{eq:x}--\eqref{eq:z} to \eqref{sys:s1} for which all equilibria apart from those in $\mathcal{C}$ are exponentially unstable by assumption. That $\mathcal{C}$ is asymptotically stable follows from Proposition \ref{col:SOT}.\end{proof}

\section{Conclusions}

\noindent This paper establishes almost global consensus on the $n$-sphere for general $n\in\N\backslash\{1\}$, a class of intrinsic gradient descent consensus protocols, and all connected, undirected graph topologies. These results show that the conditions for achieving almost global consensus are more favorable on the $n$-sphere than known results regarding other Riemannian manifolds would suggest. In particular, almost global consensus on $\mathcal{S}^1$ \cite{sarlette2009geometry} and $\SOT$ \cite{tron2012intrinsic,ohta2016attitude} requires protocols that are tailored for this specific purpose. The case of $\mathcal{S}^1$ differs from that of the general $n$-sphere due to its low dimension. There are asymptotically stable equilibrium sets on $\mathcal{S}^1$ that are disjunct from the consensus set. If these sets are embedded on the $n$-sphere for $n\in\N\backslash\{1\}$ in the form of great circles then any normal to the corresponding equatorial plane is a direction of instability. The circle can also be embedded on $\SOT$, but there it gives rise to asymptotically stable undesired equilibria. By combing our understanding of almost global consensus on $\mathcal{S}^1$ and $\mathcal{S}^2$ we design a novel class of consensus protocol on $\SOT$ which renders undesired equilibria unstable and is  shown to avoid them in simulation. 

\ifCLASSOPTIONcaptionsoff
  \newpage
\fi



%

\bibliographystyle{IEEEtran}
\bibliography{systemcontrolletters}

%
%

%

\begin{IEEEbiography}[{\includegraphics[width=1in,height=1.25in,clip,keepaspectratio]{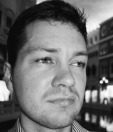}}]{Johan Markdahl} received the M.Sc. degree in Engineering Physics and Ph.D. degree in Applied and Computational Mathematics from KTH Royal Institute of Technology in 2010 and 2015 respectively. During 2010 he worked as a research and development engineer at Volvo Construction Equipment in Eskilstuna, Sweden. Currently he is a postdoctoral researcher at the Luxembourg Centre for Systems Biomedicine, University of Luxembourg.
\end{IEEEbiography}


\begin{IEEEbiography}[{\includegraphics[width=1in,height=1.25in,clip,keepaspectratio]{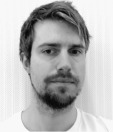}}]{Johan Thunberg} received the M.Sc. and Ph.D. degrees from KTH Royal Institute of Technology, Sweden, in 2008 and 2014, respectively. Between 2007 and 2008 he worked as a research assistant at the Swedish Defense Research agency (FOI) and between 2008 and 2009 he worked as a programmer at ENEA AB. Currently he is an AFR/FNR postdoctoral research fellow at the Luxembourg Centre for Systems Biomedicine, University of Luxembourg.
\end{IEEEbiography}

\begin{IEEEbiography}[{\includegraphics[width=1in,height=1.25in,clip,keepaspectratio]{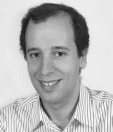}}]{Jorge Gon\c{c}alves} is currently a Professor at the Luxembourg Centre for Systems Biomedicine, University of Luxembourg and a Principal Research Associate at the Department of Engineering, University of Cambridge. He received his Licenciatura (5-year S.B.) degree from the University of Porto, Portugal, and the M.S. and Ph.D. degrees from the Massachusetts Institute of Technology, Cambridge, MA, all in Electrical Engineering and Computer Science, in 1993, 1995, and 2000, respectively. He then held two postdoctoral positions, first at the Massachusetts Institute of Technology for seven months, and from 2001 to 2004 at the California Institute of Technology with the Control and Dynamical Systems Division. At the Information Engineering Division of the Department of Engineering, University of Cambridge he was a Lecturer from 2004 until 2012, a Reader from 2012 until 2014, and since 2014 he is a Principal Research Associate. From 2005 until 2014 he was a Fellow of Pembroke College, University of Cambridge. From June to December 2010 and January to September 2011 he was a visiting Professor at the University of Luxembourg and California Institute of Technology, respectively. Since 2013 he is a Professor at the Luxembourg Centre for Systems Biomedicine, University of Luxembourg.
\end{IEEEbiography}

\end{document}

%% file: myIncludes.tex
\usepackage{subdepth} 

\usepackage{graphics} 
\usepackage{epsfig} 

\usepackage{amssymb, amsthm} 
\usepackage{amsmath,psfrag,color,graphicx} 
\usepackage{graphicx}

\usepackage[small,hang,bf,up,labelsep=period]{caption} 
\usepackage{quoting} 
\usepackage{lettrine} 
\usepackage{tocloft} 
\newcommand{\subparagraph}{}
\usepackage{titlesec}
\usepackage{titletoc}
\usepackage{fmtcount}
\usepackage{calc} 
\usepackage{cite} 
\usepackage{multicol} 
\usepackage{centernot}
\usepackage{nicefrac}
\usepackage{dsfont}
\usepackage{booktabs}

\usepackage{lipsum}
\usepackage{psfrag}
\usepackage[export]{adjustbox} 

\PassOptionsToPackage{cmyk}{xcolor}
\usepackage[cmyk]{xcolor}
\selectcolormodel{cmyk}

\usepackage{textcomp}

\usepackage{txfonts}
\usepackage{lmodern}
\usepackage{times}

\usepackage{scalerel}[2016/12/29]
\DeclareFontEncoding{LS1}{}{}
\DeclareFontSubstitution{LS1}{stix}{m}{n}
\newcommand*\laplac{%
	\text{%
		\fontencoding{LS1}%
		\fontfamily{stixfrak}%
		\fontseries{\textmathversion}%
		\fontshape{n}%
		\selectfont\symbol{"CF}}}
\makeatletter
\newcommand*\textmathversion{\csname textmv@\math@version\endcsname}
\newcommand*\textmv@normal{m}
\newcommand*\textmv@bold{b}
\makeatother
\DeclareMathOperator{\inabla}{\raisebox{1.1pt}{\scaleobj{0.85}{\laplac}}}

\newcommand{\ve}[2][]{\ensuremath{\boldsymbol{\mathrm{#2}}}_{#1}}
\newcommand{\vet}[2][]{\ensuremath{\smash{\boldsymbol{\mathrm{#2}}^{\top}_{#1}}}}
\newcommand{\vd}[2][]{\ensuremath{\dot{\boldsymbol{\mathrm{#2}}}_{#1}}}

\newcommand{\vn}[2][]{\ensuremath{\|\boldsymbol{\mathrm{#2}}_{#1}}\|}

\newcommand{\ma}[2][]{\ensuremath{\boldsymbol{\mathrm{#2}}}_{#1}}
\newcommand{\mat}[2][]{\ensuremath{\boldsymbol{\mathrm{#2}}^{\!\top}_{#1}}}
\newcommand{\md}[2][]{\ensuremath{\dot{\boldsymbol{\mathrm{#2}}}}_{#1}}

\newcommand{\R}{\ensuremath{\mathds{R}}}

\newcommand{\N}{\ensuremath{\mathds{N}}}

\newcommand{\SOT}{\ensuremath{\mathsf{SO}(3)}}

\newcommand{\SO}{\ensuremath{\mathsf{SO}(n)}}

\newcommand{\sot}{\ensuremath{\mathsf{so}(3)}}

\newcommand{\so}{\ensuremath{\mathsf{so}(n)}}

\newcommand{\Sn}{\ensuremath{\mathcal{S}^{n}}}

\newcommand{\SnN}{\ensuremath{(\mathcal{S}^{n})^N}}

\newcommand{\St}{\ensuremath{\mathcal{S}^{2}}}

\newcommand{\GC}{\ensuremath{\mathcal{S}^1}}

\newcommand{\W}{\ensuremath{\mathcal{W}}}

\newcommand{\B}{\ensuremath{\mathcal{B}}_i}

\newcommand{\ie}{\textit{i.e.}, }

\newcommand{\eg}{\textit{e.g.}, }

\newcommand{\etal}{\emph{et al.\ }}

\newcommand{\mtr}{\hspace{-0.3mm}\ensuremath{^\top}}

\newcommand{\diff}{\ensuremath{\mathrm{d}}}

\newcommand{\inv}{\ensuremath{^{-1}}}

\DeclareMathOperator{\sgn}{\mathrm{sgn}}

\newcommand\raiseT[2]{\setbox0\hbox{$#1{#2}$}\raise\dp0\box0}

\newcommand{\V}{\ensuremath{\mathcal{V}}}

\newcommand{\E}{\ensuremath{\mathcal{E}}}

\DeclareMathOperator{\acos}{\ensuremath{\mathrm{acos}}}

\DeclareMathOperator{\linspan}{\ensuremath{\mathrm{span}}}

\DeclareMathOperator{\pos}{\ensuremath{\mathrm{pos}}}

\DeclareMathOperator{\trace}{\ensuremath{\mathrm{tr}}}

\newcommand{\Ni}{\ensuremath{\mathcal{N}_i}}

\makeatletter
\newcommand{\raisemath}[1]{\mathpalette{\raisem@th{#1}}}
\newcommand{\raisem@th}[3]{\raisebox{#1}{$#2#3$}}
\makeatother

\newcommand{\ts}[2][]{\ensuremath{\mathsf{T}_{#2}#1}}

\definecolor{kthbluergb}{RGB}{25,84,166}
\definecolor{kthbluecmyk}{cmyk}{1,0.55,0,0}

\definecolor{kthblueA}{RGB}{25,84,166}
\definecolor{kthblueB}{RGB}{46,124,192}
\definecolor{kthblueC}{RGB}{112,153,209}
\definecolor{kthblueD}{RGB}{164,186,225}
\definecolor{kthblueE}{RGB}{211,220,241}

\newcommand{\um}{\hspace{0.2mm}}

\newcounter{counter} 

\newtheoremstyle{myplain} 
{\topsep}                    
{\topsep}                    
{\itshape}                   
{}                           
{\small\bfseries\sffamily}         
{.}                          
{.5em}                       
{}  

\newtheorem{theorem}[counter]{Theorem}

\newtheorem{proposition}[counter]{Proposition}

\newtheoremstyle{mydefinition} 
{\topsep}                    
{\topsep}                    
{}                   
{}                           
{\small\bfseries\sffamily}         
{.}                          
{.5em}                       
{}  

\newtheorem{definition}[counter]{Definition}
\newtheorem{example}[counter]{Example}

\newtheorem{problem}[counter]{Problem}
\newtheorem{system}[counter]{System}
\newtheorem{algorithm}[counter]{Algorithm}

\newtheoremstyle{myremark} 
{\topsep}                    
{\topsep}                    
{}                   
{}                           
{\itshape}         
{.}                          
{.5em}                       
{\thmname{#1}\thmnumber{ #2}\thmnote{.#3}}  

\newtheorem{remark}[counter]{Remark}


\newcounter{parentnumber}
